\documentclass[a4paper,leqno,10pt]{amsart}
\usepackage{amssymb}
\usepackage{amsmath,amsfonts,amssymb,upgreek}

\usepackage[breaklinks]{hyperref}
\usepackage{amscd}
\usepackage{mathtools}

\theoremstyle{thmrm}
\usepackage{tikz}
\usetikzlibrary{calc,decorations.markings}
\usepackage{tikz-cd}
\usepackage{graphicx}

\theoremstyle{plain}

\newtheorem{thm}{Theorem}[section]
\newtheorem{lemma}[thm]{Lemma}
\newtheorem{prop}[thm]{Proposition}
\newtheorem{cor}[thm]{Corollary}

\newtheorem{defn}[thm]{Definition}
\newtheorem{example}[thm]{Example}

\theoremstyle{definition}
\newtheorem{remark}[equation]{Remark}

\newcommand{\dlabel}[1]{\ifmmode \text{\ttfamily \upshape [#1] } \else
{\ttfamily \upshape [#1] }\fi \label{#1}}

\newcommand{\B}{\operatorname{B} }
\newcommand{\C}{\operatorname{C} }

\newcommand{\Ho}{\operatorname{H} }

\newcommand{\Z}{\operatorname{Z} }

\newcommand{\im}{\operatorname{Im} }

\newcommand{\Id}{\operatorname{Id}}

\newcommand{\Aut}{\operatorname{Aut} }
\newcommand{\Autb}{\operatorname{Autb} }

\newcommand{\Ext}{\operatorname{Ext} }

\newcommand{\Fun}{\operatorname{Fun} }
\newcommand{\Inn}{\operatorname{Inn} }

\newcommand{\Ker}{\operatorname{Ker} }
\newcommand{\Out}{\operatorname{Out} }

\newcommand{\IM}{\operatorname{Im} }

\title{Extensions and automorphisms of Rota-Baxter groups}

\author{Apurba Das}

\address{Department of Mathematics, Indian Institute of Technology Kharagpur, Kharagpur-721302, West Bengal, India.}

\email{apurbadas348@gmail.com,  apurbadas348@maths.iitkgp.ac.in}

\author{Nishant}

\address{School of Mathematics, Harish-Chandra Research Institute, HBNI,
Chhatnag Road, Jhunsi, Allahabad - 211 019, INDIA}

\email{nishant@hri.res.in}

    
\begin{document}

\maketitle

\begin{abstract}
The notion of Rota-Baxter groups was recently introduced by Guo, Lang and Sheng [{\em Adv. Math.} 387 (2021), 107834, 34 pp.] in the geometric study of Rota-Baxter Lie algebras. They are closely related to skew braces as observed by Bardakov and Gubarev. In this paper, we study extensions of Rota-Baxter groups by constructing suitable cohomology theories. Among others, we find relations with the extensions of skew braces. Given an extension of Rota-Baxter groups, we also construct a short exact sequence connecting various automorphism groups, which generalizes the Wells short exact sequence.

\end{abstract}

\tableofcontents

\section{Introduction}
The Rota-Baxter operators, first introduced in 1960 by G. Baxter  \cite{GB60} in the fluctuation theory of probability. Such operators can be viewed as a generalization of the integral operator on the algebra of continuous functions. In the past two decades, these operators gained great importance due to their connections with combinatorics, splitting of operads, solutions of the classical Yang-Baxter equation, Poisson geometry and mathematical physics \cite{GCR1969} \

In \cite{LHY2021}, L. Guo, H. Lang and Y. Sheng first introduced Rota-Baxter operators of weight $1$ on abstract groups. A group equipped with a Rota-Baxter operator of weight $1$ is called a Rota-Baxter group. They also considered smooth Rota-Baxter operators on Lie groups and showed that such operators can be differentiated to Rota-Baxter operators of weight $1$ on the corresponding Lie algebra. In \cite{VV2022},  V. G. Bardakov and V. Gubarev finds a connection between Rota-Baxter groups and left skew braces.
%
%
%

V. G. Bardakov and V. Gubarev showed that every Rota-Baxter operator on group induces a skew brace structure and every skew brace can be embedded inside a Rota-Baxter group.  The concept of skew brace is generalized  as weak brace, semi brace,  inverse semi brace see \cite{CMMS22}, \cite{CCS21}, \cite{CCS2021}.  A weak (left) brace is an algebraic structure $(S, +, \circ)$ where $(S, +)$ and $(S, \circ)$ are both inverse semigroup and  for all $a, b, c \in S$, the following compatibility conditions holds:
$$a \circ (b + c ) = (a \circ b) - a + (a \circ c) \hspace{10mm} \mbox{ and } \hspace{10mm} a \circ a^\prime =-a+a,$$
where $-a, a^\prime$ are inverse with respect to $+$ and $\circ$. Rota-Baxter operators on groups are further generalized to Clifford semigroups  and it has been shown that every Rota-Baxter operator on a Clifford semigroup induces a weak brace, see \cite{CMP}.  It is well known that skew braces produce non-degenerate solution of Yang-Baxter equation and  constructing Rota-Baxter operators on groups are equivalent to construct non-degenerate solutions of set-theoretical Yang-Baxter equation and these solutions need not to be non-degenerate in the case of Rota-Baxter operator on the Clifford semi-group.  Cohomology and extension theory for weighted Rota-Baxter algebra and Lie algebra is investigated in \cite{AD1}, \cite{WZ}.  A cohomology theory for relative Rota-Baxter operator is developed in \cite{JSZ}.  Cohomology theory for the extensions of  brace by an abelian group with trivial actions was investigated by L.Vendramin and V.Labed, see \cite{LV16}. A unified extension theory for non-abelian Rota-Baxter algebra and dendiform algebra by introducing non-abelian cohomology is studied in \cite{AN}. Cohomology theory for extensions of linear cycle sets have been developed by J. A. Guccione, J. J. Guccione and C. Valqui in \cite{GGV21}. Extensions with non-trivial actions, the second cohomology group and wells type exact sequence for abelian and non-abelian extensions of the skew brace are developed in \cite{DB18},\cite{NMY1}, \cite{N1}.  Extensions and Well's exact sequence for groups are investigated in very fine details, see \cite{PSY18}, \cite{JL10} The question about the extensions of Rota-Baxter operator is asked in \cite{VV2022}. In this article we answer the same.

We develop extension theory for Rota-Baxter groups, define the second cohomology group of a Rota-Baxter group acting on a abelian group and establish connections between these.  We define a faithful action of  Rota-Baxter extensions of $(H, R_H)$ by $(\Z(I), R_I)$ on  Rota-Baxter extensions of $(H, R_H)$ by $(I, R_I)$. We give a necessary and sufficient condition for inducing Rota-Baxter on semi-direct product of two Rota-Baxter groups. Further, we present an exact sequence connecting automorphism groups of  Rota-Baxter groups with second cohomology group.

\section{Rota-Baxter groups}
In this section, we recall Rota-Baxter groups and their relation with left skew braces. Our main references are \cite{LHY2021,VV2022}.

\begin{defn}
Let $(G, \cdot)$ be a group. A set map $R_G: G \rightarrow G$ is said to be Rota-Baxter operator of weight $1$ on the group $G$ if 
\begin{align*}
R_G(x) \cdot R_G(y)= R_G(x \cdot R_G(x) \cdot y \cdot R_G(x)^{-1}), \text{for all } x, y \in G.
\end{align*}
A group $(G, \cdot)$ equipped with a Rota-Baxter operator of weight $1$ is called a Rota-Baxter group. We denote a Rota-Baxter group simply by the pair $(G,R_G)$.
\end{defn}

\begin{defn}
Let $(G_1, R_{G_1})$ and $(G_2, R_{G_2})$ be two Rota-Baxter groups.  A morphism of Rota-Baxter groups from  $(G_1, R_{G_1})$ to $(G_2, R_{G_2})$ is a group homomorphism $\phi : G_1 \rightarrow G_2$ which satisfies $\phi \circ R_{G_1}=  R_{G_2} \circ \phi $.
\end{defn}

Let $(G, R_G)$ be a Rota-Baxter group. A Rota-Baxter subgroup of $(G,R_G)$ is a subgroup $H \subseteq G$ that satisfies $R_G (H) \subseteq H$. It follows that, if $H$ is a Rota-Baxter subgroup then $(H, R_{G}|_H)$ is itself a Rota-Baxter group, and the inclusion map $H \hookrightarrow G$ is a morphism of Rota-Baxter groups.


\begin{defn}
A  triple $(E,+,\circ)$, where $(E,+)$ and $(E, \circ)$ are groups is said to be a skew left brace if 
$$a \circ (b+c)=a\circ b-a+a \circ c,$$
holds for all $a,b,c \in E$, where $-a$ denotes the inverse of $a$ in $(E, +)$.
\end{defn}

Next, we state a important result established by V. Bardakov and V. Gubarev in \cite{VV2022}, which links Rota-Baxter groups and skew left braces.

\begin{thm}
Let $(G, R_G)$ be a Rota-Baxter group.  For $x, y \in G$  the operation
\begin{align}\label{gpop}
	x \circ_{R_G} y:=x \cdot R_G(x)\cdot y\cdot R_G(x)^{-1}
\end{align}
defines a group structure on $G$ and $(G, \cdot, \circ_{R_G} )$ is a skew left brace.
\end{thm}


\begin{remark}
 Every morphism of Rota-Baxter groups is a morphism of induced skew braces but converse need not be true. 
 \end{remark}
\begin{defn}
Let $(I, R_I)$ and $(H, R_H)$ be two Rota-Baxter groups.  By a \emph{Rota-Baxter extension} of $(H, R_H)$ by $(I, R_I)$,  we mean  a Rota-Baxter group $(E,R_E)$ with an exact sequence 
$$\mathcal{E} := 0 \to I \stackrel{i}{\to}  E \stackrel{\pi}{\to} H \to 0,$$ such that $i$ and $\pi$ are morphisms of Rota-Baxter groups.  We denote $i(y)=y$, which imply that $R_{E}$ restricted to $I$ is $R_I$.
By a set-theoretical section of $\mathcal{E}$  $($st-section in short$)$, we mean a map $s : H \rightarrow E$ such that $\pi s= Id_{H}$ and $s(0)=0$.
\end{defn}

\begin{remark}
 Every extension $\mathcal{E}$ of Rota-Baxter groups is an extension of skew braces $H$ by  $I$ but converse need not be true.
\end{remark}

\section{Cohomology of Rota-Baxter groups}

In this section, we investigate the relationship between modules with respect to usual group operation and group operation induced by Rota-Baxter operator. Additionally, we observe the connection between corresponding cohomologies.

Let $(H, R_H)$ be a Rota-Baxter group with group operation denoted by `$\cdot$'.  Let $I$ be a right $(H, \cdot)$-module.  For each $y \in I$ and $x \in H$, the action $ y*h =y R_H(h)$ defines an $(H, \circ_{R_H})$-module structure on $I$.  Let $C^n(H, I)$ be the set of all functions from $H^n$ to $I$ which vanishes on all degenerate tuples.

For every right $(H, \cdot)$-module $I$,  we have a cochain complex $( C^n(H, I),\delta^n )$, where coboundary map  $\delta^n : C^n(H, I) \rightarrow C^{n+1}(H, I)$ defined by 
\begin{align*}
(\delta^n f)(h_1, h_2,\ldots, h_{n+1}) = & f(h_2, \ldots,h_{n+1})+ \sum^{n}_{k=1}  (-1)^k  f(h_1, \ldots, h_k \cdot h_{k+1}, \ldots, h_{n+1})\\
& +f(h_1, \ldots,h_{n})h_{n+1}.
\end{align*}
Now, using the action of $(H, \circ_{R_H})$ on $I$, we can  define  a new cochain complex $( C^n(H, I),\partial^n )$ with the couboundary map  $\partial^n :  C^n(H, I) \rightarrow C^{n+1}(H, I)$ defined by
\begin{align*}
(\partial^n f)(h_1, h_2,\ldots, h_{n+1})= & f(h_2, \ldots,h_{n+1})+ \sum^{n}_{k=1} (-1)^k f(h_1, \ldots, h_k \circ_{R_H} h_{k+1}, \ldots, h_{n+1})\\
& +f(h_1, \ldots,h_{n})R_H(h_{n+1}).
\end{align*}
Now, consider an action of a Rota-Baxter group on an abelian Rota-Baxter group.  Let $(I, R_I)$ be an abelian Rota-Baxter group and $(H, R_H)$ be an arbitrary Rota-Baxter group such that $I$ is a right  $H$-module.  Since every group acts on its module via automorphism, therefore there exists an anti-homomorphism $\mu: H \rightarrow \Aut(I)$.  We denote $\mu(h)$ by $\mu_h$ for all $h \in H$. We assume that $\mu_h$ is an Rota-Baxter automorphism of $(I, R_I)$ for all $h \in H$. 

Let $C^{n}_{RB}(H,I):=C^n(H, I) \bigoplus C^{n-1}(H, I) \bigoplus C^{n}(H, I)$ and in particular $C^{1}_{RB}(H,I):=C^1(H, I) \bigoplus 		C^{1}(H, I)$.

Define $\delta_{RB}^{n} : C^{n}_{RB}(H,I) \rightarrow C^{n+1}_{RB}(H,I)$ by 
$$\delta_{RB}^{1}(f, g)=(\delta^1(f),  \bar{f}+ R_{I}(g), \partial^1(g)),$$
 $$\delta_{RB}^{n} (f, g, h)=(\delta^n(f), \partial^{n-1}(g)+ (-1)^{n+1}(\bar{f}-R_I(h)),  \partial^n(h)) \hspace{10mm} \textrm{ for } n\geq 2,
 $$
where $\bar{f}(h_1, \ldots, h_n)=f(R_H(h_1), \ldots, R_H(h_{n}))$ and $R_I(f)(h_1, \ldots, h_{n})=R_I(f(h_1, \ldots, h_{n+1})).$

Using the commutativity of $R_I$ and $ \mu_{R_H(h)}$, we have $\overline{\delta^n(f)}=\partial^n(\bar{f})$ and  $R_I(\partial^n(f))=\partial^n(R_I(f))$.

Now
\begin{align*}
\delta^2_{RB}( \delta^1_{RB}(f, g)) & = \delta^2_{RB}(\delta^1(f),  \bar{f}- R_{I}(g), \partial^1(g))\\
&=(\delta^2(\delta^1(f)), \partial^1(\bar{f}-R_I(g))-\overline{\delta^1(f)}+R_I(\partial^1(g)), \partial^2(\partial^1(g)))\\
&=0,
\end{align*}
 and for $n \geq 2$,
 \begin{align*}
\delta^{n+1}_{RB}( \delta^{n}_{RB}(f, g, h)) & = \delta^{n+1}_{RB}((\delta^n(f), \partial^{n-1}(g)+ (-1)^{n+1}(\bar{f}-R_I(h)),  \partial^n(h))\\
&=(0,  \partial^{n}(\partial^{n-1}(g)+ (-1)^{n+1}(\bar{f}-R_I(h)))+(-1)^{n+2}(\overline{\delta^n(f)}-R_I(\partial^n(h))), 0)\\
&= 0.
\end{align*}
Above observation shows that $(C^n_{RB}(H, I), \delta_{RB}^n)$ is a cochain complex.

Let $\sigma : (H, \circ_{R_H} ) \rightarrow \Aut(I) $ be an anti-homomorphism such that $R_I \sigma_h=\mu_{R_H(h)} R_I$. It seems that this assumption is  artificial but in extension of Rota-Baxter groups such action comes naturally. Let $\partial_{\circ}^n : C^n(H, I) \rightarrow C^{n+1}(H, I)$  be the coboundary map with respect to the  action $\sigma$ of $(H, \circ_{R_H})$ on $I$. It is easy to see that $R_I (\partial^n_{\circ}(f)) =\partial^n (R_I(f))$.

Define $\partial_{RB}^{n} : C^{n}_{RB}(H,I) \rightarrow C^{n+1}_{RB}(H,I)$ 
$$\partial_{RB}^{1}(f, g)=(\delta^1(f),  \bar{f}+ R_{I}(g), \partial_{\circ}^1(g))
,$$ 
and 
$$\partial_{RB}^{n} (f, g, h)=(\delta^n(f), \partial^{n-1}(g)+ (-1)^{n+1}(\bar{f}-R_I(h)),  \partial_{\circ}^n(h))  \hspace{10mm} \textrm{ for } n\geq 2.
$$
We have
\begin{align*}
\partial_{RB}^{2} \partial_{RB}^1(f,g) &= \partial_{RB}^{2}(\delta^1(f),  \bar{f}+ R_{I}(g), \partial_{\circ}^1(g))\\
&=(0, \partial^1(\bar{f}+ R_{I}(g))-(\bar{\delta^1(f)}-R_I(\partial^1_{\circ}(g), 0)\\
&=0.
\end{align*}
Similarly one can check that $\partial_{RB}^{n} \partial_{RB}^{n-1}(f,g, h)=0$ for $n \geq 2$. This shows that $(C^{n}_{RB}(H,I), \partial^n_{RB})$ is a cochain  complex.\

Next, we define module over a Rota-Baxter group.
\begin{defn}
	Let $I$ be an abelian group and $R_I : I \rightarrow I$  be a homomorphism. We say that $(I, R_I)$ is a right $(H, R_H)$-module if $I$ is a right $H$-module by an  action $\mu$ and the following condition holds
\begin{align}
\mu_{R_H(h)}(R_I(z))= R_I\big(\mu_{ h R_H( h)}(z  +R_I(z))  -\mu_{R_H(h)}(R_I(z))\big)
\end{align}
 for all $h \in H$ and $z \in I$.  
\end{defn} 
\begin{example}
Let $(H, R_H)$ be any Rota-Baxter group and $I$ be a trivial $H$-module that is $yh=y$ for all $y \in I$ and $h \in H$. Then $(I, R_I)$ is a right $(H, R_H)$-module for all homomorphisms $R_I : I \rightarrow I$ . 
\end{example}

\begin{example}
Let $(H, R_H)$ be a Rota-Baxter group and $I$ be a right $H$-module.  Then $(I ,R_I)$ is a $(H, R_H)$-module, where $R_I : I \rightarrow I$ is the trivial homomorphism.
\end{example}

\begin{example}
Let $I$ be a right $H$-module and $R_I(x)=-x$ then  $(I, R_I)$ is a $(H, R_H)$-module for all Rota-Baxter operators on $H$.
\end{example}

\begin{example} 
Let $R_I^2=-R_I$ and $R_I$ commutes with action of $H$ on $I$. Then   $(I, R_I)$ is a $(H, R_H)$-module.
\end{example}

Let  $(I, R_I)$ be a $(H, R_H)$-module.  Define $TC^n_{RBE}:= C^n(H, I)$ $\oplus$ $C^{n-1}(H, I)$  for $n \leq 3$.  More precisely        $TC^1_{RBE}:= C^1(H, I)$, $TC^2_{RBE}:= C^2(H, I)$ $\oplus$ $C^{1}(H, I)$ and  $TC^3_{RBE}:= C^3(H, I)$ $\oplus$ $C^{2}(H, I)$.

Define $\partial^1_{RBE} : TC^1_{RBE}(H, I) \rightarrow TC^2_{RBE}(H, I) $ by 
$$
\partial^1_{RBE}(\theta):= (\delta^1 (\theta),  -\Phi^1(\theta)),
$$ 
where $\Phi^1 (\theta(h))=R_I( \mu_{R_H(h)}(\theta(h))) -\theta(R_H(h))$.  

Define  $\partial^2_{RBE}: TC^2_{RBE} \rightarrow TC^3_{RBE}$ by
\begin{align*}
\partial^2_{RBE}(f,  g):=(\delta^2f ,\beta),
\end{align*}
and  $\beta$ is given by
\begin{align}\label{valofb}
\beta(h_1, h_2) = &  \partial^1(g)(h_1, h_2) -R_I \big(\mu_{ R_H(h_2)}(\mu_{h_2}(g(h_1))- g(h_1)\big)- \Phi^2(f)(h_1, h_2),
\end{align}
where
\begin{align*}
\Phi^2(f)(h_1, h_2)  =&  f(R_H(h_1), R_H(h_2))-R_I \big(\mu_{R_H(h_1 \circ_{R_H} h_2)}(f(h_1R_H(h_1), h_2R_H(h_1)^{-1})\\
& + \mu_{h_2 R_H(h_1)^{-1}}(f(h_1, R_H(h_1))) + f(h_2, R_H(h_1)^{-1})-f(R_H(h_1), R_H(h_1)^{-1})\big).
\end{align*}

\begin{lemma}
$\im(\partial^1_{RBE}) \subseteq \Ker(\partial^2_{RBE})$.
\end{lemma}
\begin{proof}In order to prove above lemma. It is enough to show that $\partial^2_{RBE}\partial^1_{RBE}=0$.  Let $\theta \in TC^1_{RBE}(H, I)$,  we have $$\partial^1_{RBE}(\theta)= (\delta^1 (\theta),  -\Phi^1(\theta)).
$$
Now on applying $\partial^2_{RBE}$ in the above identity, we have 
$$
\partial^2_{RBE} \big(\partial^1_{RBE}(\theta)\big)= \partial^2_{RBE}\big (\delta^1 (\theta),   -\Phi^1(\theta)\big).$$

First co-ordinate of above tupple is zero by defination of $\partial^2_{RBE} $. Writing second co-ordinate using \eqref{valofb}, we get
$$
 \partial^1(-\Phi^1(\theta))(h_1, h_2)+R_I(\mu_{R_H(h_2)}(\mu_{h_2}(-\Phi^1(h_1))+\Phi^1(h_1))\\
-\Phi^2(\delta^1(\theta))(h_1, h_2))\big).
$$
To see that the second co-ordinate is also zero, we first expand the second co-ordinate term by term.

The first term 
\begin{align}\label{first term}
 \partial^1(-\Phi^1(\theta))(h_1, h_2)= & -\big(\Phi^1(\theta)(h_2)-\Phi^1(\theta)(h_1 \circ_{R_H} h_2)+\mu_{R_H(h_2)}(\Phi^1(\theta)(h_1)\notag\\
 =&-\big(R_I( \mu_{R_H(h_2)}(\theta(h_2))) -\theta(R_H(h_2))-R_I( \mu_{R_H(h_1 \circ_{R_H} h_2)}(\theta(h_1 \circ_{R_H} h_2)))  \notag\\
 & +\theta(R_H(h_1 \circ_{R_H} h_2)) +\mu_{R_H(h_2)}(R_I( \mu_{R_H(h_1)}(\theta(h_1))) -\theta(R_H(h_1))) \big).
\end{align}

The second term
\begin{align}\label{second term}
R_I(\mu_{R_H(h_2)}(\mu_{h_2}(-\Phi(h_1))+\Phi(h_1))=&R_I(\mu_{R_H(h_2)}(\mu_{h_2}(-R_I( \mu_{R_H(h_1)}(\theta(h_1))) +\theta(R_H(h_1)))\notag\\
&+R_I( \mu_{R_H(h_1)}(\theta(h_1))) -\theta(R_H(h_1)),
\end{align}
and the third term
\begin{align}\label{third term}
\Phi^2(\delta^1(\theta))(h_1, h_2)=&  \delta^1(\theta)(R_H(h_1), R_H(h_2))-R_I \big(\mu_{R_H(h_1 \circ h_2)}(\delta^1(\theta)(h_1R_H(h_1), h_2R_H(h_1)^{-1})\notag \\
& +\mu_{h_2 R_H(h_1)^{-1}}(\delta^1(\theta)(h_1, R_H(h_1))) + \delta^1(\theta)(h_2, R_H(h_1)^{-1})-\delta^1(\theta)(R_H(h_1), R_H(h_1)^{-1})\big) \notag\\
= & \theta(R_H(h_2))-\theta(R_H(h_1 \circ_{R_H} h_2))+\mu_{R_H(h_2)}(\theta(R_H(h_1))) -R_I \big(\mu_{R_H(h_1 \circ_{R_H} h_2)}(\theta(h_2R_H(h_1)^{-1})\notag \\
&-\theta(h_1 \circ_{R_H} h_2)+\mu_{h_2R_H(h_1)^{-1}}(\theta(h_1 R_H(h_1)))  +\mu_{h_2 R_H(h_1)^{-1}}(\theta(R_H(h_1))-\theta(h_1 R_H(h_1)) \notag \\
& + \theta( R_H(h_1)^{-1})-\theta(h_2 R_H(h_1)^{-1})+ \mu_{R_H(h_1)^{-1}}(\theta(h_2 ))  -\theta( R_H(h_1)^{-1})\notag \\
& +\mu_{R_H(h_1)^{-1}}(\theta( R_H(h_1)).
\end{align}

Adding \eqref{first term}, \eqref{second term}, \eqref{third term} and using  $I$ is a $(H, R_H)$-module, it follows that  $\partial^2_{RBE} \partial^1_{RBE}=0$.  
\end{proof}

We denote $\Z^1(H,I):=\Ker(\partial^1_{RBE})$  as set of derivations   and $\Z^2(H,I):=\Ker(\partial^2_{RBE})$ as set of $2$-cocycles  from $(H, R_H)$ to $(I, R_I)$. The set $ \im (\partial^1_{RBE})$ will be called as set of $2$-coboubdaries from $H$ to $I$ and  denoted by $\B^2_{RBE}(H, I)$ More precisely

\begin{align*}
\Z_{RBE}^1(H, I)=\big\{\substack{ \lambda \hspace{1mm} \in \hspace{1mm} TC^1_{RBE} \hspace{2mm}\big| \hspace{2mm}  \lambda(h_1 h_2)= \lambda(h_2)+\mu_{h_2}(\lambda(h_1)), \lambda(R_H(h))=R_I(\mu_{R_H(h)}(\lambda(h)))\\
\forall \hspace{1mm} h_1,\hspace{1mm} h_2, \hspace{1mm}h \in \hspace{1mm}H.}   \big\}
\end{align*}
and
\begin{align*}
\Z_{RBE}^2(H, I)=\Bigg\{\substack{ (\tau, g) \hspace{1mm} \in \hspace{1mm} TC^2_{RBE} \hspace{2mm} \big| \hspace{2mm}
  \tau(h_1, h_2  h_3)+ \tau(h_2, h_3) -\tau(h_1  h_2, h_3) -\mu_{h_3}(\tau(h_1, h_2))=0,\\
  \partial^1(g)(h_1, h_2) -R_I \big(\mu_{ R_H(h_2)}(\mu_{h_2}(g(h_1))- g(h_1)\big)- \Phi^2(\tau)(h_1, h_2)=0}   \Bigg\}.
\end{align*}

We define  the second cohomology group by  $$\Ho^2_{RBE}(H, I):=\Z_{RBE}^2(H, I)/  \B^2_{RBE}(H, I)$$

 \section{Abelian Extensions of Rota-Baxter groups}
In this section, we study  abelian extensions of Rota-Baxter groups and find its connection with the second cohomology group.
 
Let $\mathcal{E} := 0 \to I \stackrel{i}{\to}  E \stackrel{\pi}{\to} H \to 0$ be a Rota-Baxter extension of   $(H, R_H)$ by $(I, R_I)$, where the  Rota-Baxter operator on $E$ is denoted by $R_E$   and $I$ is an abelian group.  Let $s : (H \rightarrow E$ be any st-section of $\mathcal{E}$.   Let $\nu:E \rightarrow \Aut(I)$, $\mu :H \rightarrow \Aut(I)$ and $\sigma :(H, \circ_{R_H}) \rightarrow \Aut(I, \circ_{R_I})$ be the maps defined by
\begin{align}\label{actions}
\nu_g(y)=& g^{-1} y g,\notag\\
\mu_h(y)= & s(h)^{-1}y s(h),\\
\sigma_h (y)= & s(h)^\prime \circ_{R_E} y \circ_{R_E} s(h),\notag
\end{align}
where $a^{\prime}$ denotes the inverse of $a$ with respect to $`\circ_{R_E}$'.  Note that $\nu$, $ \mu$ and $\sigma$ are anti-homomorphisms and  $ \mu$, $\sigma$ independent of the choice of an st-section, respectively. For any $a \in E$, there exists unique  $h \in H $ and $y \in I$ such that $a=s(h)y$.  Note that
 \begin{align}\label{RB oper ab}
 R_E(s(h)y)=&R_E(s(h)\nu^{-1}_{R_E(s(h))} (\nu_{R_E(s(h))}(y)))\notag\\
=& R_E(s(h)) R_E(\nu_{R_E(s(h))}(y)\notag\\
=& R_E(s(h)) R_I(\nu_{R_E(s(h))}(y).
\end{align}
We have $R_E(s(h))=s(\tilde{h})y_h$
 for some unique  $\tilde{h} \in H$ and $y_h \in I$.  By applying $\pi$  on the both sides of $R_E(s(h))=s(\tilde{h})y_h$, and using $\pi R_E=R_H \pi$, we have
\begin{align}
R_H(\pi(s(h)))=& \tilde{h}.
\end{align}
Hence $ R_H(h)=\tilde{h}$. Now using the value of $R_E(s(h))$ in \eqref{RB oper ab}, we get
\begin{equation}\label{RB reduced}
R_E(s(h)y)= s(R_H(h))y_h R_I(\mu_{R_H(h)}(y)).
\end{equation}
Next consider the maps $\tau, \tilde{\tau}: H \times H \rightarrow I$ given by
\begin{align} \label{cocycle 1}
\tau(h_1, h_2)=&s(h_1 h_2)^{-1}s(h_1)s(h_2) \notag,\\
 \tilde{\tau}(h_1, h_2)=&s(h_1 \circ_{R_H} h_2)^{\prime} \circ_{R_E} s(h_1) \circ_{R_E} s(h_2).
 \end{align}
We have 
\vspace{-.2cm}
\begin{align}\label{big}
R_E&  (s(h_1)y_1)  R_E(s(h_2)y_2)\notag\\
=& R_E(s(h_1)y_1 R_E(s(h_1)y_1)s(h_2)y_2 R_E(s(h_1)y_1)^{-1}) \notag \\
=&R_E(s(h_1)y_1 s(R_H(h_1))y_{h_1}R_I(\mu_{R_H(h_1)}(y_1))s(h_2)y_2R_I(\mu_{R_H(h_1)}(y_1))^{-1}y^{-1}_{h_1}  s(R_H(h_1))^{-1}) \notag \\
=&R_E(s(h_1)s(R_H(h_1)) \mu_{R_H(h_1)}(y_1)y_{h_1}R_I(\mu_{R_H(h_1)}(y_1)) s(h_2)y_2 R_I(\mu_{R_H(h_1)}(y_1))^{-1}\notag \\
& y^{-1}_{h_1} s(R_H(h_1)^{-1})   \tau(R_H(h_1), R_H(h_1)^{-1})^{-1}) \notag \\
=& R_E(s(h_1)s(R_H(h_1)) \mu_{R_H(h_1)}(y_1)y_{h_1}R_I(\mu_{R_H(h_1)}(y_1)) s(h_2)s(R_H(h_1)^{-1})\notag \\
&  \mu_{R_H(h_1)^{-1}}( y_2 R_I(\mu_{R_H(h_1)}(y_1))^{-1}y^{-1}_{h_1}) \tau(R_H(h_1), R_H(h_1)^{-1})^{-1}) \notag \\
= & R_E(s(h_1 R_H(h_1)) \tau(h_1, R_H(h_1)) \mu_{R_H(h_1)}(y_1)y_{h_1}R_I(\mu_{R_H(h_1)}(y_1)) s(h_2 R_H(h_1)^{-1})\notag \\ & \tau (h_2, R_H(h_1)^{-1}) \mu_{R_H(h_1)^{-1}}( y_2 R_I(\mu_{R_H(h_1)}(y_1))^{-1}y^{-1}_{h_1}) \tau(R_H(h_1), R_H(h_1)^{-1})^{-1}) \notag \\
=& R_E(s(h_1 R_H(h_1)) s(h_2 R_H(h_1)^{-1}) \mu_{ h_2 R_H(h_1)^{-1}}(\tau(h_1, R_H(h_1))\mu_{R_H(h_1)}(y_1)y_{h_1}\notag\\
& R_I(\mu_{R_H(h_1)}(y_1)))  \tau (h_2, R_H(h_1)^{-1}) \mu_{R_H(h_1)^{-1}}( y_2 R_I(\mu_{R_H(h_1)}(y_1))^{-1}y^{-1}_{h_1}) \tau(R_H(h_1), R_H(h_1)^{-1})^{-1}) \notag \\
=& R_E(s(h_1 R_H(h_1)h_2 R_H(h_1)^{-1}) \tau(h_1 R_H(h_1), h_2 R_H(h_1)^{-1})  \mu_{ h_2 R_H(h_1)^{-1}}(\tau(h_1, R_H(h_1)) \notag \\
& \mu_{R_H(h_1)}(y_1)y_{h_1}R_I(\mu_{R_H(h_1)}(y_1))) \tau (h_2, R_H(h_1)^{-1}) \mu_{R_H(h_1)^{-1}}( y_2 R_I(\mu_{R_H(h_1)}(y_1))^{-1}y^{-1}_{h_1})  \tau(R_H(h_1), R_H(h_1)^{-1})^{-1}) \notag \\
=& s(R_H(h_1)R_H(h_2))y_{h_1 \circ_{R_H} h_2}R_I(\mu_{R_H(h_1 \circ_{R_H} h_2)}(\tau(h_1 R_H(h_1), h_2 R_H(h_1)^{-1}) \notag \\
& \mu_{ h_2 R_H(h_1)^{-1}}(\tau(h_1, R_H(h_1)) \mu_{R_H(h_1)}(y_1)y_{h_1}R_I(\mu_{R_H(h_1)}(y_1))) \tau (h_2, R_H(h_1)^{-1})\notag\\
&    \mu_{R_H(h_1)^{-1}}( y_2 R_I(\mu_{R_H(h_1)}(y_1))^{-1}y^{-1}_{h_1})\tau(R_H(h_1), R_H(h_1)^{-1})^{-1}).  
\end{align}
On the other hand, we have
\begin{align}\label{small}
R_E&  (s(h_1)y_1) R_E(s(h_2)y_2)\notag \\
=& s(R_H(h_1)) y_{h_1} R_I(\mu_{R_H(h_1)}(y_1)) s(R_H(h_2)) y_{h_2} R_I(\mu_{R_H(h_2)}(y_2)) \notag \\
=&s(R_H(h_1)) s(R_H(h_2)) \mu_{R_H(h_2)}(y_{h_1} R_I(\mu_{R_H(h_1)}(y_1))) y_{h_2} R_I(\mu_{R_H(h_2)}(y_2)) \notag \\
=& s(R_H(h_1) R_H(h_2)) \tau(R_H(h_1), R_H(h_2)) \mu_{R_H(h_2)}(y_{h_1} R_I(\mu_{R_H(h_1)}(y_1))) y_{h_2}  R_I(\mu_{R_H(h_2)}(y_2)). 
\end{align}
Now by comparing \eqref{big} and \eqref{small}, we get
\begin{align}\label{2-cocycle condn}
 \tau& (R_H(h_1), R_H(h_2)) \mu_{R_H(h_2)}(y_{h_1} R_I(\mu_{R_H(h_1)}(y_1))) y_{h_2}\notag\\
= & \hspace{1mm} y_{h_1 \circ_{R_H} h_2}R_I(\mu_{R_H(h_1 \circ_{R_H} h_2)}(\tau(h_1 R_H(h_1), h_2 R_H(h_1)^{-1}) \mu_{ h_2 R_H(h_1)^{-1}}(\tau(h_1, R_H(h_1)) \notag \\
&  \mu_{R_H(h_1)}(y_1)y_{h_1}R_I(\mu_{R_H(h_1)}(y_1)))  \tau (h_2, R_H(h_1)^{-1})\notag \\
&  \mu_{R_H(h_1)^{-1}}( R_I(\mu_{R_H(h_1)}(y_1))^{-1}y^{-1}_{h_1})\tau(R_H(h_1), R_H(h_1)^{-1})^{-1}). 
\end{align}
Now onwards,  we write elements of $I$ additively and putting $y_1=0$ in \eqref{2-cocycle condn}, we get
\begin{align}\label{cocycle condn}
 \mu& _{R_H((h_1)}(y_{h_1})-y_{h_1 \circ_{R_H} h_2} +y_{h_2} -R_I \big(\mu_{ h_1 \circ_{R_H} h_2 R_H(h)^{-1}}(\mu_{h_2}(y_{h_1})- y_{h_1}\big)\notag \\
 =& \tau(R_H(h_1), R_H(h_2))-R_I \big( \mu_{h_1 \circ_{R_H} h_2}(\tau(h_1R_H(h_1), h_2R_H(h_1)^{-1})\notag \\
 &+\mu_{h_2 R_H(h_1)^{-1}}(\tau(h_1, R_H(h_1))
  + \tau(h_2, R_H(h_1)^{-1})-\tau(R_H(h_1), R_H(h_1)^{-1})\big).
\end{align}
By \eqref{2-cocycle condn} and  \eqref{cocycle condn},  we observe  that   $\mu$ and $R_I$ together satisfy the following compatibility condition :
\begin{align}\label{comp condn}
\mu_{R_H(h_2)}(R_I(\mu_{R_H(h_1)}(y)))=& R_I\big(\mu_{ h_2 R_H( h_2)}(\mu_{R_H(h_1)}(y)+ R_I(\mu_{R_H(h_1)}(y)))\notag \\
&-\mu_{R_H(h_2)}(R_I(\mu_{R_H(h_1)}(y)))\big),
\end{align}
for all $h_1, h_2 \in H$ and $y \in I$. This shows that $(I, R_I)$ is a $(H, R_H)$-module.

In the view of  relation  $R_E(s(h))=s(R_H(h)))y_h$, we define a map $g$ from $H$ to $I$  by 
\begin{align}\label{RB map}
 g(h)=y_h.
\end{align}
Clearly $g$  is  well-defined and \eqref{cocycle condn}  can be rewritten as
\begin{align}\label{new cocycle condn}
 \partial& ^1(g)(h_1, h_2) -R_I \big(\mu_{ h_1 \circ_{R_H} h_2 R_H(h)^{-1}}(\mu_{h_2}(g(h_1))- g(h_1)\big)\notag \\
 =& \tau(R_H(h_1), R_H(h_2))-R_I \big(\mu_{h_1 \circ_{R_H} h_2}(\tau(h_1R_H(h_1), h_2R_H(h_1)^{-1})
 +\mu_{h_2 R_H(h_1)^{-1}}(\tau(h_1, R_H(h_1)))\notag \\
 &  + \tau(h_2, R_H(h_1)^{-1})-\tau(R_H(h_1), R_H(h_1)^{-1})\big).
\end{align}
Note that
\begin{align}\label{circ cocycle}
R_E(s(h_1))R_E(s(h_2)) =& R_E(s(h_1) \circ_{R_E} s(h_2))\notag\\
=&R_E\big(s(h_1 \circ_{R_H} h_2) \circ_{R_E} \tilde{\tau}(h_1, h_2)\big)\notag \\
=&s(R_H(h_1 \circ_{R_H} h_2)) g(h_1 \circ_{R_H} h_2) R_I(\tilde{\tau}(h_1, h_2)).
\end{align} 
Comparing \eqref{small}  and \eqref{circ cocycle} with $y_1=y_2=0$, we have the following relation among   $\tilde{\tau}$, $\tau$ and $g$, 
\begin{align}\label{mixed}
 \partial^1(g)(h_1, h_2)=&R_I( \tilde{\tau}(h_1, h_2))- \tau(R_H(h_1), R_H(h_2)).
 \end{align}
This shows  that  $(\tau, g, \tilde{\tau}) \in \Ker(\partial^2_{RB})$.

 Let $s_1$ and $s_2$ be two st-sections of $\mathcal{E}$.  We have $s_2(h)=s_1(h)z_h$ for some unique $z_h \in I$.  Define $\theta: H \rightarrow I$ by $\theta(h):=z_h$. Clearly, $\theta$ is a well-defined map. Let $\tau_1, \tau_2$ be the 2-cocycles corresponding to $s_1$ and $s_2$,  respectively.  The following relation is well-known 
\begin{align} \label{cocycle relation}
-\theta(h_1  h_2) + \tau_1(h_1, h_2)+ {\mu}_{h_2}(\theta(h_1)) + \theta(h_2)=\tau_2(h_1, h_2).
\end{align}
Let $R_{E}(s_1(h))=s_1(R_H(h)) \prescript{}{1}{y}_{h}$ and $R_{E}(s_2(h))=s_2(R_H(h)) \prescript{}{2}{y}_{h}$. Then  we have
\begin{align*}
 s_1(R_H(h)) \theta(R_H(h)) \prescript{}{2}{y}_{h}= R_E(s_2(h))= R_E(s_1(h)\theta(h))= s_1(R_H(h))\prescript{}{1}{y}_{h} R_I(\mu_{R_H(h)}(\theta(h))).
\end{align*}
Thus,
\begin{equation}\label{RB cocycle diff}
 \prescript{}{2}{y}_{h} - \prescript{}{1}{y_{h}}=  R_I(\mu_{R_H(h)}(\theta(h))) -\theta(R_H(h)).
\end{equation}
If we define the maps $g_1, g_2 : H \rightarrow I$ corresponding to $s_1$ and $s_2$ by $g_i(h)=\prescript{}{i}{y}_{h}$, then \eqref{RB cocycle diff} can be rewritten as
\begin{align} \label{g condn}
g_2(h)-g_1(h)=R_I(\mu_{R_H(h)}(\theta(h))) -\theta(R_H(h)).
\end{align}
\begin{prop}\label{kernel}
Let $I$ be an abelian group and $$\mathcal{E} := 0 \to I \stackrel{i}{\to}  E \stackrel{\pi}{\to} H \to 0$$ be a Rota-Baxter extension of   $(H, R_H)$ by $(I, R_I)$.  For any st-section $s$, let   $\tau$ and $g$ be the maps as defined in \eqref{cocycle 1} and  \eqref{RB map}, respectively.  Then $(\tau, g) \in  \Ker(\partial^2_{RBE})$. Moreover, the cohomology class of $( \tau, g)$ is independent of the choice of any st-section of $\mathcal{E}$.
\end{prop}
\begin{proof}
It follows from  \eqref{new cocycle condn} that the pair $(\tau, g) \in  \Ker(\partial^2_{RBE})$.  Let $s_1$ and $s_2$ be two st-sections of $\mathcal{E}(\tau_1, g_1)$ and $(\tau_2, g_2)$ be the  pairs corresponding to $s_1$ and $s_2$.  Then \eqref{cocycle relation} and \eqref{g condn} together implies that $(\tau_1, g_1)$ and $(\tau_2, g_2)$ are cohomologous. This completes the proof.
\end{proof}

\begin{prop}\label{correspondence}
Let $(I, R_I)$ be an $(H, R_H)$-module via action $
\mu$.  Let $(\tau, g) \in \Ker(\partial^2_{RBE})$  then 
the map $R: H \times I \rightarrow H \times I$ given by 
\begin{align}\label{RB extn}
R(h, y):=\big( R_H(h), g(h)+R_I\big( \mu_{R_H(h)}(y )\big)\big)
\end{align}
defines a Rota-Baxter operator on the group extension $H \times I$ of $H$ by $I$  and $(H \times I, R)$ defines a Rota-Baxter extension of $(H, R_H)$ by $(I, R_I)$, where group operation on $H \times I$ is given by
\begin{align}\label{group extension}
(h_1, y_1)+(h_2, y_2):=\big(h_1h_2,  y_1+\mu_{h_2}(y_2)+\tau(h_1, h_2)\big).
\end{align} 
Moreover, cohomologous $2$-cocycles define equivalent extensions. We denote such extension by $\mathcal{E}(\tau, g)$.
\end{prop}
\begin{proof} From the theory of group extensions it follows that the operation in  \eqref{group extension} defines a group extension of $H$ by $I$ for details see \cite[Chapter 2]{PSY18}.  Now we verify that the operator $R$ defined in \eqref{RB extn} is a Rota-Baxter operator on $H \times I$. 
\begin{align}\label{RB proof}
R(h_1, y_1) R(h_2, y_2)=& \big( R_H(h_1), g(h_1)+R_I\big( \mu_{R_H(h_1)}(y_1 )\big)\big)\big( R_H(h_2), g(h_2)+R_I\big( \mu_{R_H(h_2)}(y_2 )\big)\big)\notag\\
=& \big( R_H(h_1) R_H(h_2), g(h_1)+R_I\big( \mu_{R_H(h_1)}(y_1 )\big)+\mu_{R_H(h_2)}\big(g(h_2)\notag \\
&+R_I(\mu_{R_H(h_2)}(y_2 ))+ \tau(R_H(h_1),R_H(h_2))\big).
\end{align}
Using the fact that $(\tau, g) \in \Ker(\partial^2_{RBE})$, the expression of  \eqref{RB proof} becomes
\begin{align}\label{RB proof 2}
&\big(R_H(h_1 \circ_{R_H} h_2) ,  R_I\big( \mu_{R_H(h_1)}(y_1 )\big)+g(h_1 \circ_{R_H} h_2) +R_I \big(\mu_{ R_H(h)^{-1} R_H (h_1 \circ_{R_H} h_2)}(\mu_{h_2}g(h_1)- g(h_1)\big)\notag \\
&-R_I(\mu_{R_H (h_1 \circ_{R_H} h_2)}\big(\tau(h_1R_H(h_1), h_2R_H(h_1)^{-1}) + \mu_{h_2 R_H(h_1)^{-1}}(\tau(h_1, R_H(h_1))\notag \\
& + \tau(h_2, R_H(h_1)^{-1})-  \tau(R_H(h_1)^{-1}, R_H(h_1))\big) +R_I\big( \mu_{R_H(h_2)}(y_2 )) \big).
\end{align} 
Using \eqref{comp condn},  \eqref{RB proof 2} become 
\begin{align*}
&=R\Big((h_1, y_1) R(h_1, y_1) (h_2, y_2)  R(h_1, y_1)^{-1} \Big).
\end{align*}
This shows that $R$ defined in \eqref{RB extn} is a Rota-Baxter operator on $H \times I$.  It is easy to verify that $R i= i R_I $ and $R \pi=\pi R_H$,  which shows that $$ \mathcal{E}(\tau, g):=0 \to I \stackrel{i}{\to}  H \times I \stackrel{\pi}{\to} H \to 0$$ is a Rota-Baxter extension, where $i$ and $\pi$ are natural injection and projection, respectively. 

Next, we show that extensions corresponding to cohomologous $2$-cocycles are equivalent. Let $\mathcal{E}(\tau_1, g_1)$ and $\mathcal{E}(\tau_2, g_2)$ be two extensions corresponding to cohomologous $2$-cocycles $(\tau_1, g_1)$ and $(\tau_2, g_2)$, respectively.  By the definition of $\Ho_{RBE}^2(H, I)$ there exists a map $\theta: H \rightarrow I$ such that  
\begin{align*}
\tau_2(h_1, h_2)-\tau_1(h_1, h_2)=\delta^1(\theta)(h_1, h_2)
\end{align*}
and
\begin{align*} 
g_2(h)-g_1(h)=\theta(R_H(h))-R_I(\mu_{R_H(h)}(\theta(h))),
\end{align*}
for all $h_1, h_2, h \in H$.

Define $ \varsigma: \mathcal{E}(\tau_1, g_1) \rightarrow \mathcal{E}(\tau_2, g_2)$ by $ \varsigma(h, y)=(h, y+\theta(h)).$   It follows that the map $ \varsigma$ is an isomorphism of groups.  We will now show that $ \varsigma$ is a morphism of Rota-Baxter groups. Let $R_1$ and $R_2$ be  the Rota-Baxter operator on $\mathcal{E}(\tau_1,  g_1)$ and $\mathcal{E}(\tau_2, g_2)$ defined via \eqref{RB extn}. We have
\begin{align*}
 \varsigma\big( R_1(h, y)\big) &= \varsigma\big( R_H(h), g_1(h)+R_I\big( \mu_{R_H(h)}(y )\big)\big)\\
 &=\big( R_H(h), g_1(h)+R_I\big( \mu_{R_H(h)}(y )\big)+\theta(R_H(h))\big)\\
 &=\big( R_H(h), g_2(h)+R_I\big( \mu_{R_H(h)}(y+\theta(h) )\big)  \big)\\
 &=R_2 \big( \varsigma (h, y) \big).
\end{align*}
This shows that $\varsigma$ is a morphism of Rota-Baxter groups. 
\end{proof}

\begin{thm}\label{bij}
Let $(I, R_I)$ be an $(H, R_H)$-module and $\Ext_{\mu}(H, I)$ be set of equivalence classes of all extensions of   $(H, R_H)$ by $(I, R_I)$ with action  $\mu$. Then, $\Ext_{\mu}(H, I)$ is in bijection with $\Ho^{2}_{RBE}(H, I)$.
\end{thm}

\begin{proof}Let $$\mathcal{E} := 0 \to I \stackrel{i}{\to}  E \stackrel{\pi}{\to} H \to 0$$ be a Rota-Baxter extension of  $(H, R_H)$ by $(I, R_I)$.  Let $s$ be a st-section of $\mathcal{E}$.  Let $\tau$ and $g$ be maps as defined in \eqref{cocycle 1} and  \eqref{RB map}, respectively. Define
\begin{align}
&\psi : \Ext_{\mu}(H, I) \rightarrow \Ho^{2}_{RBE}(H, I) \hspace{.2cm} \mbox{by} \notag\\
&\psi(\mathcal{E}):=[(\tau, g)]
\end{align}
$[(\tau, g)]$ denotes the cohomology class of $(\tau, g)$. The map $\psi$ is well-defined is follows from Proposition \ref{kernel}.  Conversely, define
\begin{align*}
&\omega : \Ho^{2}_{RBE}(H, I) \rightarrow \Ext_{\mu}(H, I) \hspace{.2cm}\mbox{by} \\
&\omega([\tau, g]):= [\mathcal{E}(\tau, g)]
\end{align*}
where $[\mathcal{E}(\tau, g)]$ represents the equivalence class of $\mathcal{E}(\tau, g)$ in $\Ho^{2}_{RBE}(H, I)$.  To show that $\omega$ is well-defined, it is enough to show that $2$-cocycles corresponding to equivalent extensions are cohomologous. Let $E_1$ and $E_2$  be two equivalent Rota-Baxter extensions of $(H,R_H)$ by $(I, R_I)$ then we have
the following commutative diagram
$$\begin{CD}
0 @>>> I @>i>> E_1 @>{{\pi} }>> H @>>> 0\\ 
&& @V{\text{Id}} VV@V{\phi} VV @V{\text{Id} }VV \\
0 @>>> I @>i^\prime>>  E_2  @>{{\pi^\prime} }>> H  @>>> 0
\end{CD}$$
where $\phi: E_1 \rightarrow E_2$ is isomorphism of Rota-Baxter groups.  Let $s_1$ and $s_2$ be st-sections corresponding to the extensions $E_1$ and $E_2$,  respectively.  Let $(\tau_1, g_1)$ and $(\tau_2, g_2)$ be 2-cocycles corresponding to $s_1$ and $s_2$, respectively. Let $s^{\prime}= \phi s_1$ then $s^{\prime}$ is an section of $E_2$. Let $(\tau^\prime, g^\prime)$ be $2$-cocycle corresponding to $s^\prime$ then $(\tau^\prime, g^\prime)=(\tau_1, g_1)$ and by Proposition \ref{kernel}, $(\tau^\prime, g^\prime)$ and $(\tau_2, g_2)$ are cohomologous. Hence $(\tau_1, g_1)$ and $(\tau_2, g_2)$ are cohomologous,  which shows that $\omega$ is well-defined.

Next we prove that $ \psi$ and $ \omega$ are inverse of each other.  In order to show that $\psi \omega= Id_{ \Ho^{2}_{RBE}(H, I)}$, we will prove that $(\tau, g)$ is cohomologous to a 2-cocycle  $(\tau^\prime, g^\prime)$ corresponding to some st-section of $\mathcal{E}(\tau, g).$ Let $s: H \rightarrow \mathcal{E}(\tau, g)$ given by $s(h)=(h, 0)$. Clearly, $s$ is a st-section of $\mathcal{E}(\tau, g)$.  It can be easily seen that  the $2$-cocycle corresponding to st-section $s$ is cohomologous to $(\tau, g)$.  On the other side $\omega \psi= \Id_{\Ext_{\mu}(H, I)}$ follows from proposition \ref{correspondence}. This completes the proof.
\hfill   $\Box$

\end{proof}

\begin{remark}
 Let $(I, R_I)$ be a $(H, R_H)$-module.  We know that every element of $ \Ho^{2}_{RBE}(H, I)$ defines a Rota-Baxter extension and every Rota-Baxter extension is extension of induced skew left braces.  Let $ \Ho^{2}_{N}(H, I)$ be the second cohomology group of skew left brace   $H$ with cofficients in $I$ (defined in \cite{NMY1}). There is a natural embedding $ \Ho^{2}_{RBE}(H, I)$ into $ \Ho^{2}_{N}(H, I)$. More precisely, for $[(\tau, g)] \in \Ho^{2}_{RBE}(H, I) $ construct an extension and corresponding to that extension define $\tau$ and $\tilde{\tau}$ as defined in  \eqref{cocycle 1}, then  $[(\tau, g)] \rightarrow [(\tau, \tilde{\tau})]$ is the required embedding.
\end{remark}

\begin{remark}
If we consider $\mu_h=Id_H$ for all $h \in H$ $($in case of centeral extensions$)$. then the maps $\partial^1_{RBE}$ and $\partial^2_{RBE}$ could be simplified as $
\partial^1_{RBE}(\theta):= (\delta^1 (\theta),  -\Phi^1(\theta)),$ 
where $\Phi^1 (\theta(h))=R_I(\theta(h)) -\theta(R_H(h))$
and $\partial^2_{RBE}(f,  g):=(\delta^2f ,\beta).$ In this case  $\beta$  and $\Phi^2(f)$ is given by
\begin{align*}
\beta(h_1, h_2) =&  \partial^1(g)(h_1, h_2)- \Phi^2(f)(h_1, h_2),\notag\\
\Phi^2(f)(h_1, h_2)=& f(R_H(h_1), R_H(h_2))-R_I \big(f(h_1R_H(h_1), h_2R_H(h_1)^{-1})
+f(h_1, R_H(h_1))\\
& + f(h_2, R_H(h_1)^{-1})-f(R_H(h_1), R_H(h_1)^{-1})\big).
\end{align*}
We also have the following commutative diagram
\[ \begin{tikzcd}
C^1 \arrow{r}{\Phi^1} \arrow[swap]{d}{\delta^1} & C^1 \arrow{d}{\partial^1} \\%
C^2 \arrow{r}{\Phi^2}& C^2
\end{tikzcd}
\]

that is $\partial^1 \Phi^1=\Phi^2 \delta^1$.
\end{remark}
\noindent \textbf{Problem 1.} Is it possible to generalize the maps $\Phi^1$ and $\Phi^2$  such that $\partial^n \Phi^n=\Phi^{n+1} \delta^n$ for $n \geq 3$?

\section{General  Extensions of Rota-Baxter groups}
In this section, we examine general extensions of Rota-Baxter groups.  We also define an action of abelian extensions on non-abelian extensions.

 Let $\mathcal{E} := 0 \to I \stackrel{}{\to}  E \stackrel{\pi}{\to} H \to 0$ be a Rota-Baxter extension of $(H, R_H)$ by $(I, R_I)$.  Let $s : H \rightarrow E$ be a st-section of $\mathcal{E}$.  Let $\nu, \mu, \sigma$ be as in \eqref{actions}.
Note that $\nu$ is an anti-homomorphism but $\mu$ and $\sigma$ in general need not be anti-homomorphisms. However they satisfy the following identity
\begin{align}
\mu_{h_1 h_2} =&i_{\tau(h_1, h_2){-1}} \mu_{h_2} \mu_{h_1},\notag\\
\sigma _{h_1 \circ_{R_H} h_2} =&i^{\circ}_{\tilde{\tau}(h_1, h_2)^{-1}} \sigma_{h_2} \sigma_{h_1},
\end{align}
where $i_x$ and $i^{\circ}_x$ are inner automorphisms of $I$ and $(I, \circ_{R_I})$, respectively, for  $x \in I$.

We observed that $R_E(s(h))=s(R_H(h))y_h$ for some unique $y_h \in I$ and  for any $a \in E$ there exists unique  $h \in H $ and $y \in I$ such that $a=s(h)y$.  Note that
 \begin{align}\label{RB oper}
 R_E(s(h)y)=&R_E(s(h)\nu_{R_E(s(h))^{-1}} (\nu_{R_E(s(h))}(y)))\notag\\
=& R_E(s(h)) R_E(\nu_{R_E(s(h))}(y)\notag\\
=& R_E(s(h)) R_I(\nu_{(s(R_H(h))y_h)}(y)\notag\\
=&s( R_H(h)) y_h R_I(i_{y^{-1}_h}(\mu_{R_H(h)}(y)))).
\end{align}
Using that $R_E$ is a Rota-Baxter operator on $E$, we have
\begin{align*}
& R_E  (s(h_1)y_1)  R_E(s(h_2)y_2)\notag\\
=& R_E(s(h_1)y_1 R_E(s(h_1)y_1)s(h_2)y_2 R_E(s(h_1)y_1)^{-1}) \notag \\
=&R_E\big(s(h_1)y_1 s(R_H(h_1))y_{h_1}R_I(i_{y^{-1}_{h_1}}(\mu_{R_H(h_1)}(y_1)))s(h_2)y_2 R_I(i_{y^{-1}_{h_1}}(\mu_{R_H(h_1)}(y_1)))^{-1}\notag \\
& y^{-1}_{h_1}s(R_H(h_1)^{-1}) \tau(R_H(h_1), R_H(h_1)^{-1})^{-1} \big)  \notag \\
=&R_E\big(s(h_1)s(R_H(h_1)) \mu_{R_H(h_1)}(y_1) y_{h_1}R_I(i_{y^{-1}_{h_1}}(\mu_{R_H(h_1)}(y_1)))s(h_2)s(R_H(h_1)^{-1})\notag \\
& \mu_{R_H(h_1)^{-1}} ( y_2 R_I(i_{y^{-1}_{h_1}}(\mu_{R_H(h_1)}(y_1)))^{-1}  y^{-1}_{h_1}) \tau(R_H(h_1), R_H(h_1)^{-1})^{-1} \big)\notag \\
=& R_E\big(s(h_1 R_H(h_1)) \tau(h_1, R_H(h_1)) \mu_{R_H(h_1)}(y_1) y_{h_1}R_I(i_{y^{-1}_{h_1}}(\mu_{R_H(h_1)}(y_1)))s(h_2 R_H(h_1)^{-1})\notag \\
& \tau(h_2, R_H(h_1)^{-1}) \mu_{R_H(h_1)^{-1}} ( y_2 R_I(i_{y^{-1}_{h_1}}(\mu_{R_H(h_1)}(y_1)))^{-1}  y^{-1}_{h_1}) \tau(R_H(h_1), R_H(h_1)^{-1})^{-1} \big)\notag \\
=& R_E\big(s(h_1 R_H(h_1)) s(h_2 R_H(h_1)^{-1}) \mu_{h_2 R_H(h_1)^{-1}} \big(\tau(h_1, R_H(h_1)) \mu_{R_H(h_1)}(y_1) y_{h_1}\notag \\
& R_I(i_{y^{-1}_{h_1}}(\mu_{R_H(h_1)}(y_1)))\big)\tau(h_2, R_H(h_1)^{-1}) \mu_{R_H(h_1)^{-1}} ( y_2 R_I(i_{y^{-1}_{h_1}}(\mu_{R_H(h_1)}(y_1)))^{-1}  y^{-1}_{h_1})\notag \\
& \tau(R_H(h_1), R_H(h_1)^{-1})^{-1} \big)\notag \\
\end{align*}
\begin{align}\label{ NA big}
=& R_E\Big(s(h_1 R_H(h_1) h_2 R_H(h_1)^{-1}) \tau(h_1 R_H(h_1), h_2 R_H(h_1)^{-1}) \mu_{h_2 R_H(h_1)^{-1}} \big(\tau(h_1, R_H(h_1))\notag \\
& \mu_{R_H(h_1)}(y_1) y_{h_1} R_I(i_{y^{-1}_{h_1}}(\mu_{R_H(h_1)}(y_1)))\big)  \tau(h_2, R_H(h_1)^{-1})  \mu_{R_H(h_1)^{-1}}\big ( y_2 R_I(i_{y^{-1}_{h_1}}\notag \\
& (\mu_{R_H(h_1)}(y_1)))^{-1}  y^{-1}_{h_1} \big)  \tau(R_H(h_1), R_H(h_1)^{-1})^{-1} \Big)\notag \\
=& s(R_H(h_1) R_H(h_2)) y_{h_1 \circ_{R_H} h_2} R_I \Big( i^{-1}_{y_{h_1 \circ h_2}} \big( \mu_{R_H(h_1) R_H(h_2)}(\tau(h_1 R_H(h_1), h_2 R_H(h_1)^{-1})\notag \\
& \mu_{h_2 R_H(h_1)^{-1}} \big(\tau(h_1, R_H(h_1)) \mu_{R_H(h_1)}(y_1) y_{h_1} R_I(i_{y^{-1}_{h_1}}(\mu_{R_H(h_1)}(y_1)))\big)\tau(h_2, R_H(h_1)^{-1})\notag \\
& \mu_{R_H(h_1)^{-1}}\big ( y_2 R_I(i_{y^{-1}_{h_1}}(\mu_{R_H(h_1)}(y_1)))^{-1}  y^{-1}_{h_1} \big) \tau(R_H(h_1), R_H(h_1)^{-1})^{-1} \big) \Big).
\end{align}
We also have another method to expand this expression
\begin{align}\label{NA small}
& R_E  (s(h_1)y_1) R_E(s(h_2)y_2)\notag \\
=& s(R_H(h_1)) y_{h_1} R_I(i_{y^{-1}_{h_1}}(\mu_{R_H(h_1)}(y_1))) s(R_H(h_2)) y_{h_2} R_I(i_{y^{-1}_{h_2}}(\mu_{R_H(h_2)}(y_2))) \notag \\
=&s(R_H(h_1)) s(R_H(h_2)) \mu_{R_H(h_2)}\big(y_{h_1} R_I(i_{y^{-1}_{h_1}}(\mu_{R_H(h_1)}(y_1)))\big) y_{h_2} R_I(i_{y^{-1}_{h_2}}(\mu_{R_H(h_2)}(y_2))) \notag \\
=& s(R_H(h_1) R_H(h_2)) \tau(R_H(h_1), R_H(h_2))  \mu_{R_H(h_2)}\big(y_{h_1} R_I(i_{y^{-1}_{h_1}}(\mu_{R_H(h_1)}(y_1)))\big) y_{h_2} \notag \\
& R_I(i_{y^{-1}_{h_2}}(\mu_{R_H(h_2)}(y_2))). 
\end{align}
By comparing \eqref{ NA big} and \eqref{NA small}, we get
\begin{align}\label{NA parent}
& y_{h_1 \circ_{R_H} h_2}R_I \Big( i^{-1}_{y_{h_1 \circ_{R_H} h_2}} \big( \mu_{R_H(h_1) R_H(h_2)}(\tau(h_1 R_H(h_1), h_2 R_H(h_1)^{-1}) \mu_{h_2 R_H(h_1)^{-1}} \big(\tau(h_1, R_H(h_1))\notag \\
& \mu_{R_H(h_1)}(y_1) y_{h_1} R_I(i_{y^{-1}_{h_1}}(\mu_{R_H(h_1)}(y_1)))\big)\tau(h_2, R_H(h_1)^{-1})\notag \\
& \mu_{R_H(h_1)^{-1}}\big ( y_2 R_I(i_{y^{-1}_{h_1}}(\mu_{R_H(h_1)}(y_1)))^{-1}  y^{-1}_{h_1} \big) \tau(R_H(h_1), R_H(h_1)^{-1})^{-1} \big) \Big)=A.
\end{align}
where $A$ is given by 
\begin{align*}
A=& \tau(R_H(h_1), R_H(h_2))  \mu_{R_H(h_2)}\big(y_{h_1} R_I(i_{y^{-1}_{h_1}}(\mu_{R_H(h_1)}(y_1)))\big) y_{h_2} R_I(i_{y^{-1}_{h_2}}(\mu_{R_H(h_2)}(y_2))). 
\end{align*}
Also, we have
\begin{align}\label{circ condn}
 R_E(s(h_1))R_E(s(h_2))= & R_E(s(h_1) \circ_{R_E} s(h_2))\notag \\
= & R_E(s(h_1 \circ_{R_H} h_2) \circ_{R_I} \tilde{\tau}(h_1,  h_2))\notag\\
=& R_E(s(h_1 \circ_{R_H} h_2)) R_E(\tilde{\tau}(h_1, h_2))\notag\\
=& s(R_H(h_1) R_H(h_2)) y_{h_1 \circ_{R_H} h_2} R_I(\tilde{\tau}(h_1, h_2)).
\end{align}
By putting $y_1=y_2=0$ in \eqref{ NA big} and comparing with \eqref{circ condn},  we get
\begin{align}\label{circ to old}
R_I(\tilde{\tau}(h_1, h_2)) =& R_I \Big( i^{-1}_{y_{h_1 \circ_{R_H} h_2}} \big( \mu_{R_H(h_1) R_H(h_2)}(\tau(h_1 R_H(h_1), h_2 R_H(h_1)^{-1})\notag \\
& \mu_{h_2 R_H(h_1)^{-1}} \big(\tau(h_1, R_H(h_1)) y_{h_1} \big)\tau(h_2, R_H(h_1)^{-1})\notag \\
& \mu_{R_H(h_1)^{-1}}\big ( y^{-1}_{h_1} \big) \tau(R_H(h_1), R_H(h_1)^{-1})^{-1} \big) \Big).
\end{align}
We have a well-defined map $g: H \rightarrow I$ given by 
\begin{align}\label{g map}
g(h)=y_h
\end{align}
and \eqref{NA parent} and \eqref{circ to old} can be rewritten as
\begin{align}\label{NA parent 2}
& g(h_1 \circ_{R_H} h_2)R_I \Big( i^{-1}_{g(h_1 \circ_{R_H} h_2)} \big( \mu_{R_H(h_1) R_H(h_2)}(\tau(h_1 R_H(h_1), h_2 R_H(h_1)^{-1}) \mu_{h_2 R_H(h_1)^{-1}} \big(\tau(h_1, R_H(h_1))\notag \\
& \mu_{R_H(h_1)}(y_1) g(h_1) R_I(i_{g(h_1)^{-1}}(\mu_{R_H(h_1)}(y_1)))\big)\tau(h_2, R_H(h_1)^{-1})\notag \\
& \mu_{R_H(h_1)^{-1}}\big ( y_2 R_I(i_{g(h_1)^{-1}}(\mu_{R_H(h_1)}(y_1)))^{-1} g(h_1)^{-1} \big) \tau(R_H(h_1), R_H(h_1)^{-1})^{-1} \big) \Big)=A,
\end{align}
 where $A$ is given by
\begin{align*}
A=& \tau(R_H(h_1), R_H(h_2))  \mu_{R_H(h_2)}\big(g(h_1) R_I(i_{g(h_1)^{-1}}(\mu_{R_H(h_1)}(y_1)))\big) g(h_2)  R_I(i_{g(h_2)^{-1}}(\mu_{R_H(h_2)}(y_2))),
\end{align*}
and 
\begin{align}\label{coho re}
R_I(\tilde{\tau}(h_1, h_2))=& R_I \Big( i^{-1}_{g(h_1 \circ_{R_H} h_2)} \big( \mu_{R_H(h_1) R_H(h_2)}(\tau(h_1 R_H(h_1), h_2 R_H(h_1)^{-1})\notag \\
& \mu_{h_2 R_H(h_1)^{-1}} \big(\tau(h_1, R_H(h_1)) g(h_1) \big)\tau(h_2, R_H(h_1)^{-1})\notag \\
& \mu_{R_H(h_1)^{-1}}\big (g(h_1)^{-1} \big) \tau(R_H(h_1), R_H(h_1)^{-1})^{-1} \big) \Big),
\end{align}
respectively.
It follows from calculations that  $\tau$ satisfy 
\begin{equation}\label{cocycle 2}
\tau(h_1, h_2  h_3) \tau(h_2, h_3) \tau(h_1  h_2, h_3)^{-1}  (\mu_{h_3}(\tau(h_1, h_2)))^{-1}=0.
\end{equation}

 If we have a Rota-Baxter extension $\mathcal{E} := 0 \to I \stackrel{}{\to}  E \stackrel{\pi}{\to} H \to 0$ and a st-section $s$ then we can construct a triplet $(\mu, \tau, g)$ corresponding to $s$ as described in \eqref{actions},  \eqref{cocycle 1}
and \eqref{g map}. 

Next we see that how the triplets corresponding to different sections are related. Let $s_1$ and $s_2$ be two different st-sections of $\mathcal{E}$.  We have $s_2(h)=s_1(h)z_h$ and let $R_E(s_i(h)) =s_i(R_H(h)\prescript{}{i}{y}_{h}$ for some unique $z_h, \prescript{}{1}{y}_{h} $ and $\prescript{}{2}{y}_{h} \in I$.  Define the maps $g_i$ and $\theta $ from $H$ to $I$ by $g_i(h)=\prescript{}{i}{y}_{h}$ and  $\theta: H \rightarrow I$ by $\theta(h):=z_h$.  Let $\prescript{}{i}{\mu}$ be action of $H$ on $I$ corresponding to $s_i$ as defined in \eqref{actions}. Then we have 
 \begin{align}\label{st action }
 \prescript{}{2}{\mu}_h &= i_{\theta(h)^{-1}} \prescript{}{1}{\mu}_h.
\end{align}  
 Let $\tau_1, \tau_2$ be 2-cocycles corresponding to $s_1$ and $s_2$,  respectively.  The following relation is well known from theory of group extensions
\begin{align} \label{NA cocycle relation}
\theta(h_1  h_2)^{-1}  \tau_1(h_1, h_2) \prescript{}{1}{\mu}_{h_2}
(\theta(h_1))  \theta(h_2)=\tau_2(h_1, h_2).
\end{align}
We see that how $g_1$ and $g_2$ are related
\begin{align}\label{first g}
	R_E(s_2(h)) = s_2(R_H(h)) g_2(h) =s_1(R_H(h)) \theta(R_H(h)) g_2(h).
\end{align}
We also have,
\begin{align}\label{second g}
R_E(s_2(h)) = R_E(s_1(h) \theta(h))=s_1(R_H(h)) g_1(h) R_I\big(i_{g_{1}(h)^{-1}}(\prescript{}{1}{\mu}_{R_H(h)}(\theta(h)))\big).
\end{align}

By comparing \eqref{first g} and \eqref{second g}, we have
\begin{align}\label{NA g condn}
	\theta(R_H(h)) g_2(h)  &=g_1(h) R_I\big(i_{g_{1}(h)^{-1}}(\prescript{}{1}{\mu}_{R_H(h)}(\theta(h)))).
\end{align}

As we mentioned that $\mu$ defined in \eqref{actions}  is, in general, not be a homomorphism.   However, clearly $\mu$ composed with the natural projection
$$\Aut(I) \rightarrow \Out(I):=\Aut(I)/ \Inn(I)$$ 
is a homomorphism, denoted by 
$\bar{\mu} : H \rightarrow \Out(I)$.

Now, we state a well-known theorem of group extensions,  which stands true for Rota-Baxter extensions of groups.
\begin{thm}
The homomorphism $\bar{\mu}: H \rightarrow \Out(I)$ is independent of choice of any st-section $s : H \rightarrow E$. Furthermore, if $\prescript{}{1}{\bar{\mu}}$, $\prescript{}{2}{\bar{\mu}}$ are the homomorphisms associated with two equivalent Rota-Baxter extensions of $(H, R_H)$ by $(I, R_I)$,  then $\prescript{}{1}{\bar{\mu}}=\prescript{}{2}{\bar{\mu}}$.
\end{thm}
In the view of proceeding theorem, for every equivalence class of extensions $[\mathcal{E}] \in Ext(H,I)$ there is a unique homomorphism $\bar{\mu}: H \rightarrow \Out(I),$ called the coupling associated to $[\mathcal{E}]$. 
Let $(H, R_H)$ and $(I, R_I)$ be two Rota-Baxter groups and $\bar{\mu}: H \rightarrow \Out(I)$ a coupling. We set
$$\Ext_{\bar{\mu}}(H, I)=\{[\mathcal{E}]\hspace{.1cm} \vert \hspace{.1cm}\bar{\mu} \mbox{ is the coupling associated to } \mathcal{E} \}.$$
Thus, we can write
\vspace{-.01cm}
$$   \Ext(H, I)=\bigsqcup_{\bar{\mu}}\Ext_{\bar{\mu}}(H, I).$$
\begin{defn} 
Let $(H, R_H)$ and $(I, R_I)$ be two Rota-Baxter groups.  Let $\mu : H \rightarrow \Aut(I),$ $ \tau: H \times H \rightarrow I,$
and $g: H \rightarrow I $ be the maps  that  satisfy \eqref{st action }, \eqref{NA parent 2} and \eqref{cocycle 2}.  We call the triplet $(\mu, \tau, g)$ to be an associated triplet corresponding to action of $H$ on $I$.
\end{defn}

\begin{thm}
Every $(\mu, \tau,g)$ be an associated triplet corresponding to action of $H$ on $I$ defines a Rota-Baxter extension of $(H, R_H)$ by $(I, R_I)$ denoted by  $\mathcal{E}(\mu, \tau,g)$.  Furthermore, if $\mathcal{E}(\mu, \tau,g)$ is a Rota-Baxter extension of $(H, R_H)$ by $(I, R_I)$  defined by $( \mu, \tau, g)$, then there exists a st-section of $\mathcal{E}(\mu, \tau,g)$ for which the  associated triplet is $( \mu, \tau, g)$.
\end{thm}

\begin{proof}
Let $( \mu,\tau, g)$ be an associated triplet. Define $R: H \times_{\mu, \tau} I \rightarrow H \times_{\mu, \tau} I$ by 
$$R(h, y)=\big( R_H(h), g(h)R_I (i_{g(h)^{-1}}(\mu_{R_H(h)}(y))) \big).$$
Here $H \times I$ is a group extension of $H$ by $I$ defined by  $\mu$ and $\tau$.  More precisely, the group structure  on $H \times_{\mu, \tau} I$ is given by
$$(h,_1 y_1)(h_2, y_2)=(h_1 h_2, \tau(h_1, h_2) \mu_{h_2}(y_1)y_2).$$
The map $R$ defined above is a Rota-Baxter operator  is follows from \eqref{NA parent 2}. It is easy to see that 
$$\mathcal{E}(\mu, \tau,g) := 0 \to I \stackrel{i}{\to} H \times I\stackrel{\pi}{\to} H \to 0,$$ 
is a Rota-Baxter extension, where $i$ and $\pi$ are natural injection and projection, respectively. Let $ s : H \rightarrow H \times I $ be  $s(h):=(h, 0)$, where $0$ denotes the identity of $I$. Then  $s$ is a  st-section of $\mathcal{E}(\mu, \tau,g)$ and the triplet corresponding to $s$ is $(\mu, \tau, g).$
\hfill   $\Box$

\end{proof}
 Let $\alpha$ be a homomorphism from $H$ to $\Out(I)$.  Define  $$\mathcal{Z}^2_{\alpha}(H, I):=\{(\mu, \tau, g) \hspace{1mm}  \vert \hspace{1mm} \bar{\mu}=\alpha \mbox{ and } (\mu, \tau, g) \mbox{ is an associated triplet}\}.$$
Next we see relationship between the associated triplets which defines the equivalent extension. Let $(\prescript{}{1}{\mu}, \tau_1, g_1)$ and 
$(\prescript{}{2}{\mu}, \tau_2, g_2)$ be two elements of $\mathcal{Z}^2_{\alpha}(H, I)$.  We say that $(\prescript{}{1}{\mu}, \tau_1, g_1) \sim (\prescript{}{2}{\mu}, \tau_2, g_2) $  if the extensions $\mathcal{E}(\prescript{}{1}{\mu}, \tau_1, g_1)$ and  $\mathcal{E}(\prescript{}{2}{\mu}, \tau_2, g_2)$ are equivalent Rota-Baxter extensions.

\begin{prop}
The relation  `$\sim$' on $\mathcal{Z}^2_{\alpha}(H, I)$ is an equivalence relation.
\end{prop}

\begin{thm}
Two associated triplets $(\prescript{}{1}{\mu}, \tau_1, g_1)$ and  $(\prescript{}{2}{\mu}, \tau_2, g_2)$ defines equivalent extensions if and only if there exists a map $\theta: H \rightarrow I$ such that $(\prescript{}{1}{\mu}, \tau_1, g_1)$ and  $(\prescript{}{2}{\mu}, \tau_2, g_2)$ with $\theta$ satisfies \eqref{st action },\eqref{NA cocycle relation} and \eqref{NA g condn}.
\end{thm}

\begin{proof}
Let $\theta:  H \rightarrow I$ be a map that satisfies \eqref{st action },\eqref{NA cocycle relation} and \eqref{NA g condn}.  Define 
$$ \psi :  \mathcal{E}(\prescript{}{2}{\mu}, \tau_2, g_2) \rightarrow  \mathcal{E}(\prescript{}{1}{\mu}, \tau_1, g_1) \mbox{ by }$$ 
$$\psi(h,y):=(h,  \theta(h) y ).$$
Claim is that $\psi$ is an isomorphism of Rota-Baxter groups and the diagram commutes
$$\begin{CD}
0 @>>> I @>i_2>> \mathcal{E}(\prescript{}{2}{\mu}, \tau_2, g_2) @>{{\pi_2} }>> H @>>> 0\\ 
&& @V{\text{Id}} VV@V{\psi} VV @V{\text{Id} }VV \\
0 @>>> I @>i_1>>  \mathcal{E}(\prescript{}{1}{\mu}, \tau_1, g_1)  @>{{\pi_1} }>> H  @>>> 0.
\end{CD}$$
\textit{Proof of the claim }: Let $R_1$ and $R_2$ be the Rota-Baxter operators corresponding to $\mathcal{E}(\prescript{}{1}{\mu}, \tau_1, g_1)$  and $\mathcal{E}(\prescript{}{2}{\mu}, \tau_2, g_2)$, respectively.  For $h_1, h_2 \in H$ and $y_1, y_2 \in I $, we have 
\begin{align*}
\psi\big( (h_1, y_1)(h_2, y_2) \big)&=\psi\big(h_1 h_2, \tau_2(h_1, h_2) \prescript{}{2}{\mu}_{h_2}(y_1) y_2  \big)\notag \\
&=\big( h_1h_2,  \theta(h_1 h_2)\tau_2(h_1, h_2) \prescript{}{2}{\mu}_{h_2}(y_1) y_2  \big)\notag \\
&=\big (h_1 h_2,  \tau_1(h_1, h_2) \prescript{}{1}{\mu}_{h_2}(\theta(h_1)) \theta(h_2)\prescript{}{2}{\mu}_{h_2}(y_1) y_2  \big)\notag\\
&=\big (h_1 h_2,  \tau_1(h_1, h_2) \prescript{}{1}{\mu}_{h_2}(\theta(h_1)) \prescript{}{1}{\mu}_{h_2}(y_1) \theta(h_2) y_2  \big)\notag \\
&=\psi(h_1, y_1) \psi(h_2, y_2).
\end{align*}
Next we have 
\begin{align}\label{RBmor}
\psi R_2(h, y)&=\psi\big( R_H(h), g_2(h)R_I (i_{g_2(h)^{-1}}(\mu_{R_H(h)}(y))) \big)\notag \\
&= \big( R_H(h), \theta(R_H(h)) g_2(h)R_I (i_{g_2(h)^{-1}}(\mu_{R_H(h)}(y))) \big).
\end{align}
Using   \eqref{NA g condn} in \eqref{RBmor}, we get that $\psi$ is a morphism of Rota-Baxter groups.  Proving $\psi$ is bijection and the above diagram  commutes is a simple verification.

\noindent \textit{Conversely}: Let $(\prescript{}{1}{\mu}, \tau_1, g_1)$ and  $(\prescript{}{2}{\mu}, \tau_2, g_2)$ be two associated triplets such that there exist an isomorphism $\psi:\mathcal{E}(\prescript{}{2}{\mu}, \tau_2, g_2) \rightarrow \mathcal{E}(\prescript{}{1}{\mu}, \tau_1, g_1) $ and the diagram commutes
$$\begin{CD}
0 @>>> I @>i_2>> \mathcal{E}(\prescript{}{2}{\mu}, \tau_2, g_2) @>{{\pi_2} }>> H @>>> 0\\ 
&& @V{\text{Id}} VV@V{\psi} VV @V{\text{Id} }VV \\
0 @>>> I @>i_1>>  \mathcal{E}(\prescript{}{1}{\mu}, \tau_1, g_1)  @>{{\pi_1} }>> H  @>>> 0.
\end{CD}$$
Let $h \in H $ and $y \in I$. Using the commutativity of the above diagram we have 
$\psi i_2 = i_1 $ and $\pi_2 \psi= \pi_1$, which implies that $\psi(e,y)=(e, y)$ and $\psi(h, 0)=(h, g_h)$ for some unique $g_h \in I$. 
Define $\theta: H \rightarrow I$ by $\theta(h):=g_h$. It is easy to see that $\theta$ is the required map. 
\hfill   $\Box$

\end{proof}

\begin{cor}
Let $(\prescript{}{1}{\mu}, \tau_1, g_1)$ and  $(\prescript{}{2}{\mu}, \tau_2, g_2)$ be two elements of 
$\mathcal{Z}^2_{\alpha}(H, I)$.  Then $(\prescript{}{1}{\mu}, \tau_1, g_1)$ $\sim$  $(\prescript{}{2}{\mu}, \tau_2, g_2)$ if there exists a  map $\theta: H \rightarrow I $ such that $(\prescript{}{1}{\mu}, \tau_1, g_1)$ and  $(\prescript{}{2}{\mu}, \tau_2, g_2)$ satisfies \eqref{st action }, \eqref{NA cocycle relation} and \eqref{NA g condn}.
\end{cor}
\begin{defn}
Define $$\mathcal{H}^2_{\alpha}(H, I):=\mathcal{Z}^2_{\alpha}(H, I)/ \sim$$ and denote $[(\mu,\tau, g)] \in \mathcal{H}^2_{\alpha}(H, I)$, the equivalence class of $(\mu,\tau, g)$. 
\end{defn}

\begin{thm}
There exist a bijection between  $\mathcal{H}^2_{\alpha}(H, I)$ and $\Ext_{\alpha}(H, I).$
\end{thm} 
\begin{proof}
 The map $\phi:  \mathcal{H}^2_{\alpha}(H, I) \rightarrow \Ext_{\alpha}(H, I) $ defined by $\phi([(\mu, \tau, g)]):=[\mathcal{E}(\mu, \tau,g)]$ is well defined and injective follows from definition of `$\sim$'. The map is surjective follows from the construction of $\mu, \tau$ and $g$ as defined in \eqref{actions}, \eqref{cocycle 1} and \eqref{g map}, respectively.
\end{proof}
\begin{prop}
Let $(H, R_H)$ and $(I, R_I)$ be Rota-Baxter groups such that  $\Z(I)$ is a Rota-Baxter subgroup of $(I, R_I)$.  Let $(\mu, \tau, g)$ be an associated triplet corresponding to action of $H$ on $I$,  then $(\Z(I), R_I)$ is an $(H, R_H)$-module with action $\tilde{\mu}: H \rightarrow \Aut(\Z(I))$, where $\tilde{\mu}_h:=\mu_h \vert_{Z(I)}$.  Moreover, every element of  $\mathcal{Z}^2_{\alpha}(H, I)$ induces same action on $\Z(I)$.  
\end{prop}
\begin{proof}
By putting $h_1=0$ in  \eqref{NA parent 2}, we get an expression independent of $\tau$. Now using  that $y_1, y_2,  \in \Z(I)$ and $\Z(I)$ is invariant under $R_I$ and $\mu_h$ for $h \in H$, we can easily establish that $(\Z(I), R_I)$ is a  $(H, R_H)$-module via action $\tilde{\mu}$.
\end{proof}

\begin{thm}\label{cocycle change}
Let $(\mu,\tau, g)$ be an associated triplet and $\mu^\prime $ be an action of $H$ on $I$ with $\overline{\mu}=\overline{\mu^\prime}$, then there exist $\tau^\prime : H\times H \rightarrow I$ and $g^\prime : H \rightarrow I $ such that $(\mu^\prime,\tau^\prime, g^\prime)$ is an associated triplet and $[(\mu,\tau, g)]=[(\mu^\prime,\tau^\prime, g^\prime)]$.
\end{thm}
\begin{proof}
The condition $\overline{\mu}=\overline{\mu^\prime}$ implies that there exists a map $\theta : H \rightarrow I $ such that $\mu_h=i_{\theta(h)} \mu^\prime_h$.  Let $s_1$ be a st-section of $\mathcal{E}(\mu, \tau, g)$ inducing $(\mu, \tau, g)$. Define $s_2(h)=s_1(h) \theta(h)$ for $ h \in H$. Then the action induced by $s_2$ is $\mu^\prime$ and the $2$-cocycle $(\tau^\prime, g^\prime)$ corresponding to $s_2$ are the recquired maps.  As $s_1$ and $s_2$ are different sections of same extension so they will induce equivalent extensions which implies that $[(\mu,\tau, g)]=[(\mu^\prime,\tau^\prime, g^\prime)]$.
\hfill   $\Box$

\end{proof}

\begin{thm}
Let $(I, R_I)$ be a Rota-Baxter group such that $\Z(I)$ is a Rota-Baxter subgroup of $I$. Then there exists a  free action of $\Ho^2(H, \Z(I))$ on $\mathcal{H}^2_{\alpha}(H, I)$ where $(\Z(I), R_I)$ is a $(H, R_H)$-module with respect to the  action induced by elements of $\mathcal{Z}^2_{\alpha}(H, I)$.  
\end{thm}
\begin{proof}
Let $[(\tau^\prime, g^\prime)] \in \Ho^2(H, \Z(I))$ and $[(\mu, \tau, g)] \in \mathcal{H}^2_{\alpha}(H, I)$.  Define
 \begin{align}\label{freeact}
 [(\tau^\prime , g^\prime)] [(\mu, \tau, g)]:=[(\mu, \tau \tau^\prime ,  g g^\prime)],
\end{align}
  where 
$\tau \tau^\prime (h_1, h_2):=\tau(h_1, h_2) \tau^\prime (h_1, h_2)$ and $g g^\prime (h):=g(h) g^\prime(h)$, for all $h_1, h_2 , h \in H$. It is easy to see that $(\mu, \tau \tau^\prime ,  g g^\prime) \in \mathcal{Z}^2_{\alpha}(H, I) $.  Let $(\tau_1^\prime, g_1^\prime), (\tau_2^\prime, g_2^\prime) \in \ker(\partial^2_{RBE})$ and $( \prescript{}{1}{\mu}_h , \tau_1, g_1), ( \prescript{}{2}{\mu}_h , \tau_2, g_2) \in \mathcal{Z}^2_{\alpha}(H, I)$ such that $[(\tau_1^\prime, g_1^\prime)]=[(\tau_2^\prime, g_2^\prime)]$  and $[( \prescript{}{1}{\mu}_h , \tau_1, g_1)]=[( \prescript{}{2}{\mu}_h , \tau_2, g_2)]$. We claim that $[( \prescript{}{1}{\mu}_h , \tau_1 \tau^\prime_1, g_1 g^\prime_1)] =[( \prescript{}{2}{\mu}_h , \tau_2 \tau^\prime_2, g_2 g^\prime_2)]$.  As  $[(\tau_1^\prime, g_1^\prime)]=[(\tau_2^\prime, g_2^\prime)]$, there exists  a map $\theta$ such that $(\tau_1^\prime, g_1^\prime)$ and $(\tau_2^\prime, g_2^\prime)$ are cohomologous by $\theta$.  Also, there exists a map $\theta^\prime$ such that $( \prescript{}{1}{\mu}_h , \tau_1, g_1) \sim ( \prescript{}{2}{\mu}_h , \tau_2, g_2)$ by $\theta^\prime$.  Straightforward calculations shows that 
$( \prescript{}{1}{\mu}_h , \tau_1 \tau^\prime_1, g_1 g^\prime_1) \sim ( \prescript{}{2}{\mu}_h , \tau_2 \tau^\prime_2, g_2 g^\prime_2)$ by $\theta \theta^\prime$, which proves the claim. This show that the action is well-defined.

Next we show that the defined action is free.  Let $[(\tau^\prime, g^\prime)] [(\mu, \tau, g)]=[(\mu, \tau \tau^\prime ,  g g^\prime)]=[(\mu, \tau  ,  g )]$, then we have $\theta : H \rightarrow I$ such that $(\mu, \tau \tau^\prime ,  g g^\prime) \sim  (\mu, \tau  ,  g )$  by $\theta$.  It follows from \eqref{st action } that $\theta(h) \in \Z(I)$ for all $h \in H$ and  \eqref{NA cocycle relation}, \eqref{NA g condn} proves that $[(\tau^\prime, g^\prime)]$ is the identity of $\Ho^2(H, \Z(I))$, which shows that the action is free. 
\hfill   $\Box$

\end{proof}
\noindent \textbf{Problem 2.} Under what condition the  action \eqref{freeact} becomes transitive?

\section{Split extensions of Rota-Baxter groups}
In this section, we study the special case where short exact sequence of Rota-Baxter groups splits as a sequence of groups and provide some examples.

 Let $(H, R_H)$ and $(I, R_I)$ be Rota-Baxter groups and $\mathcal{E} := 0 \to I \stackrel{i}{\to}  E \stackrel{\pi}{\to} H \to 0$ be  an extension Rota-Baxter groups in which $\mathcal{E}$ splits as an extension of groups. That is, there exists a unique st-section $s: H  \rightarrow E$ such that $s$ is a group homomorphism. 
By putting $\tau(h_1, h_2)=0$ in \eqref{NA parent 2}, we get
\begin{align}\label{split condn final}
& \mu_{R_H(h_2)}\Big(g(h_1)R_I\big(i_{g(h_1)^{-1}}(\mu_{R_H(h_1)}(y_1))\big)\Big) g(h_2)R_I\big(i_{g(h_2)^{-1}}(\mu_{R_H(h_2)}(y_2))\big)\notag \\
&= g(h_1 \circ_{R_H} h_2) R_I\Big( i_{g(h_1 \circ_{R_H} h_2)^{-1}} \big( \mu_{R_H(h_1 \circ_{R_H} h_2)}(z)\big) \Big),
\end{align}
where  $z$ is given by
 \begin{align}\label{final z condn}
z=& \mu_{R_H(h_1) h_2 R_H(h_1)^{-1} }\Big(R_I\big(i_{g(h_1)^{-1}}(\mu_{R_H(h_1)}(y_1))\big)\Big)   \mu_{h_2 R_H(h_1)^{-1}}\bigg(y_2 \Big(R_I\big(i_{g(h_1)^{-1}}(\mu_{R_H(h_1)}(y_1))\big) \Big)^{-1} g(h_1)^{-1} \bigg).
\end{align}
\begin{thm}\label{splitgrp}
Let $(H, R_H)$ and $(I, R_I)$ be two Rota-Baxter groups. Let $\mu$ be an anti-homomorphism from $H$ to $Aut(I)$ and $g$ is a map from $H$ to $I$  such that $g(0)=0$ and together  $\mu$ and $g$ satisfy \eqref{split condn final} for all $h_1, h_2 \in H$ and $y_1, y_2 \in I$. Then the map $R: H \times I \rightarrow H \times I$ given by
\begin{align}\label{split RB operator}
R(h, y)=\Big(R_H(h),  g(h)  R_I\big(i_{g(h)^{-1}}(\mu_{R_H(h)}(y))\big)\Big)
\end{align}
defines a Rota-Baxter operator on the semi-direct product $H \ltimes_{\mu} I$, where the group operation in $H \ltimes_{\mu} I$ is given by
\begin{align*}
(h_1, y_1) (h_2, y_2) =&(h_1 h_2, \mu_{h_2}(y_1) y_2).
\end{align*}
\end{thm}
\begin{proof}
To show that $R$ defined in \eqref{split RB operator} is a Rota-Baxter operator, we expand both sides one by one. Consider, 
\begin{align}\label{first split condn}
& R(h_1, y_1)R(h_2, y_2)\notag \\
=& \Big(R_H(h_1),  g(h_1)  R_I\big(i_{g(h_1)^{-1}}(\mu_{R_H(h_1)}(y_1))\big)\Big)\Big(R_H(h_2),  g(h_2)  R_I\big(i_{g(h_2)^{-1}}(\mu_{R_H(h_2)}(y_2))\big)\Big)\notag \\
=& \Big(R_H(h_1)R_H(h_2),  \mu_{R_H(h_2)}\big(g(h_1) R_I(i_{g(h_1)^{-1}}(\mu_{R_H(h_1)}(y_1)))\big) g(h_2)  R_I\big(i_{g(h_2)^{-1}}(\mu_{R_H(h_2)}(y_2))\big)\Big).
\end{align}
On the other side, we have
\begin{align}\label{second condn split}
&R\big((h_1, y_1)R(h_1, y_1)(h_2, y_2)R(h_1, y_1)^{-1}\big)\notag \\
=&R\Big((h_1, y_1)\big(R_H(h_1),  g(h_1)  R_I(i_{g(h_1)^{-1}}(\mu_{R_H(h_1)}(y_1))\big)(h_2, y_2)\notag \\
&(R_H(h_1)^{-1},  \mu_{R_H(h_1)^{-1}}\big(g(h_1)  R_I(i_{g(h_1)^{-1}}(\mu_{R_H(h_1)}(y_1)))\big)^{-1} \Big)\notag\\
=&R\Big(h_1 R_H(h_1), \mu_{R_H(h_1)}\big(g(h_1)  R_I(i_{g(h_1)^{-1}}(\mu_{R_H(h_1)}(y_1)))\big)\notag\\
&(h_2 R_H(h_1)^{-1}, \mu_{R_H(h_1)^{-1}}(y_2)\mu_{R_H(h_1)^{-1}}\big(g(h_1)  R_I(i_{g(h_1)^{-1}}(\mu_{R_H(h_1)}(y_1)))\big)^{-1}\Big)\notag\\ 
=& R\Big(h_1R_H(h_1)h_2 R_H(h_1)^{-1},\mu_{h_2R_H(h_1)^{-1}}\big(\mu_{R_H(h_1)}(g(h_1)  R_I(i_{g(h_1)^{-1}}(\mu_{R_H(h_1)}(y_1)))))\big)\notag \\
& \mu_{R_H(h_1)^{-1}}(y_2)\mu_{R_H(h_1)^{-1}}\big(g(h_1)  R_I(i_{g(h_1)^{-1}}(\mu_{R_H(h_1)}(y_1)))\big)^{-1}\Big)\notag\\
=&\Big(R_H(h_1 \circ_{R_H} h_2),  g(h_1 \circ_{R_H} h_2)R_I(i_{g(h_1 \circ_{R_H} h_2)^{-1}}\big(\mu_{R_H(h_1 \circ_{R_H} h_2)}(\mu_{h_2R_H(h_1)^{-1}}\big(\mu_{R_H(h_1)}(g(h_1) \notag \\
& R_I(i_{g(h_1)^{-1}}(\mu_{R_H(h_1)}(y_1)))))))\Big).
\end{align}
Using \eqref{split condn final} in \eqref{second condn split},  it follows   that  \eqref{first split condn} and \eqref{second condn split} are equal. This proves that the map defined in \eqref{split RB operator} is a Rota-Baxter operator.
We denote this  Rota-Baxter group by $(H, I, \mu, g, R)$.
\hfill   $\Box$

\end{proof}

\begin{thm}\label{splthm}
Let $(H, R_H)$ and $(I, R_I)$ be two Rota-Baxter groups.  Then the Rota-Baxter group $(H, I, \mu, g, R)$ defines a Rota-Baxter extension of $(H, R_H)$ by $(I, R_I)$ such that it splits as an extension of groups.
\end{thm}
\begin{proof} 
Let $$\mathcal{E}(\mu, g):= 0 \to I \stackrel{i}{\to}  (H, I, \mu, g, R) \stackrel{\pi}{\to} H \to 0$$
be a sequence of groups,  where $i$ and $\pi$ are the natural injection and projection, respectively.
We have $$R(i(y))=R(0,y)= (0,R_I(y))  =i(R_I(y)),$$  and 
$$\pi(R(h,y))=\pi \Big(R_H(h),  g(h)  R_I\big(i_{g(h)^{-1}}(\mu_{R_H(h)}(y))\big)\Big)=R_H(\pi(h, y)),$$
for all $h \in H$ and $y \in I$. This shows that $\mathcal{E}(\mu, g)$ is a Rota-Baxter extension.  Let $s: H \rightarrow (H, I, \mu, g, R)$ given by $s(h)=(h, 0)$. It is easy to verify that $s$ is a  st-section of $\mathcal{E}(\mu, g)$ and $s$ is also a group homomorphism.  This shows that $\mathcal{E}(\mu, g)$ as an extension of groups splits.
\hfill   $\Box$

\end{proof}
\begin{thm}
Let $\mathcal{E}:  0 \to I \stackrel{i}{\to}  E \stackrel{\pi}{\to} H \rightarrow 0$ be an extension of Rota-Baxter groups such that it splits as an extension of groups.  Then $\mathcal{E}$ is equivalent to $\mathcal{E}(\mu, g)$ for some $g: H \rightarrow I$ and anti-homomorphism $\mu : H \rightarrow \Aut(I)$.
\end{thm}
\begin{proof}
Let $\mathcal{E}:  0 \to I \stackrel{i}{\to}  E \stackrel{\pi}{\to} H \rightarrow 0$ be an extension of Rota-Baxter groups such that it splits as an extension of groups. Define $\mu$ and $g$ corresponding to $s$ as defined in \eqref{actions} and \eqref{g map}.   Then we can define a Rota-Baxter extension $\mathcal{E}(\mu, g)$ of $(H, R_H)$ by $(I, R_I)$ following Theorem \ref{splthm}.  It is easy to verify that  the map $$\phi : (E, R_E) \rightarrow  \mathcal{E}(\mu, g)$$ given by $\phi(s(h) y)=(h, y)$ is an isomorphism of Rota-Baxter groups and the following diagram commutes
$$\begin{CD}
0 @>>> I @>i>> E @>{{\pi} }>> H @>>> 0\\ 
&& @V{\text{Id}} VV@V{\phi} VV @V{\text{Id} }VV \\
0 @>>> I @>i^\prime>>   \mathcal{E}(\mu, g) @>{{\pi^\prime} }>> H  @>>> 0.
\end{CD}$$
\hfill   $\Box$

\end{proof}
\begin{thm}
	Let  $\mathcal{E}_1 := 0 \to I \stackrel{}{\to}  E \stackrel{\pi_1}{\to} H \to 0,$ be an extension of skew left braces corresponding to some fixed skew left brace structure on $I$, $E$ and $H$ such that additive groups of $I$, $E$ and $H$ are complete groups. Then there exist Rota Baxter operators  $R_I, R_E$ and $R_H$ on  $I$, $E$ and $H$, respectively, such that the sequence $\mathcal{E}_2 := 0 \to(I, R_I) \stackrel{}{\to}  (E, R_E) \stackrel{\pi_1}{\to} (H, R_H) \to 0$ defines  an extension of Rota-Baxter groups and the skew brace extension defined by $\mathcal{E}_2$ is same as $\mathcal{E}_1$.
\end{thm}

\begin{proof}
	Let $R_E$ be a Rota-Baxter operator on $E$ which induces skew left  brace structure on $E$, that is under consideration in  $\mathcal{E}_1$.  As $I$ is an ideal of the skew left brace $E$ that means $I$ is itself a skew left brace. For $ x, y \in I$, we have
	\begin{align*}
		x \circ_{R_E}  y =&  x R_E(x)  y R_E(x)^{-1}.
	\end{align*}
	As $I$ is a complete group and every skew brace structure on complete group is induced by a  Rota-Baxter operator, thus we have $R_I : I \rightarrow  I$ which induces the same skew brace structure on $ I.$ We have 
	\begin{align*}
		x \circ_{R_I} y =& x \circ_{R_E}  y,\\
		R_I(x)  y R_I(x)^{-1}=& R_E(x) y R_E (x)^{-1}.
	\end{align*}
	Since, the  above equation is true for all $x, y \in I$,  we have $R_I(x)R_E( x)^{-1} \in \Z( I)$, which is trivial.  Hence $R_I(x)=R_E(x)$ for all $x \in I$, this shows that $ I$ is invariant under $R_E$. Now from $\mathcal{E}_1$   we have $E/I \cong H$ as skew left brace and skew left brace structure on $E/I$ is induced by Rota-Baxter operator $\overline{R}_E: E/I \rightarrow E/I$, defined by 
	$ \overline{R}_E(\overline{x})=\overline{R_E(x) }$, where $\overline{x}$ denotes the image of $x \in E$ in $E/I$  under natural projection. Let skew left brace structure on $H$ is induced by some Rota-Baxter operator $R_H$. We have
	\begin{align*}
		\overline{\pi}_1 (\overline{x_1 \circ_{\overline{R}_E} x_2})= & \overline{\pi}_1 (\overline{x}_1 \circ_{\overline{R}_E} \overline{x}_2),\\
		\overline{\pi}_1(\overline{x_1 R_E(x_1) x_2 R_E(x_1)^{-1} }) = & \overline{\pi}_1 (\overline{x}_1) \circ_{R_H}  \overline{\pi}_1( \overline{x}_2),\\
		\overline{\pi}_1(\overline{x}_1 ) )\overline{\pi}_1(\overline{R_E(x_1)}) \overline{\pi}_1(\overline{x_2}) \overline{\pi}_1(\overline{R_E(x_1)})^{-1}=&  \overline{\pi}_1 (\overline{x}_1) R_H(\overline{\pi}_1 (\overline{x}_1)) \overline{\pi}_1 (\overline{x}_2) R_H(\overline{\pi}_1 (\overline{x}_1))^{-1}.
	\end{align*}
	Using that additive group of skew left brace  $H$ is complete group, we have $\pi_1 R_E=R_H \pi_1$. This shows that $\mathcal{E}_2 := 0 \to(I, R_I) \stackrel{}{\to}  (E, R_E) \stackrel{\pi_1}{\to} (H, R_H) \to 0$ is an extension of Rota-Baxter groups.
	\hfill $\Box$

\end{proof}
Next, we give some examples of Rota-Baxter operators on group which comes through the extensions.  We have used  GAP \cite{GAP}  to compute these examples.

\begin{example}
 The symmetric group $S_3$ over $\{1, 2,3 \}$ have total $8$ Rota-Baxter operators and all of them are extension of homomorphisms $($Rota-Baxter operator is just homomorphism over abelian group$)$  on $I=\langle(1,2,3) \rangle $ and $H=S_3/I  $, which is isomorphic to cyclic group of order $2$. Here we list all of them except the trivial operator \

$1)$ $R_1(x)=e$ for all $x \in I $ and $R_1(y)=(2,3)$ for all $y \notin I$.\

$2)$ $R_2(x)=e$ for all $x \in I $ and $R_2(y)=(1,3)$ for all $y \notin I$.\

$3)$ $R_3(x)=e$ for all $x \in I $ and $R_3(y)=(1,2)$ for all $y \notin I$.\

$4)$ $R_4(x)=x^2$ for all $x \in I $ and $R_4(2,3)=e, R_4(1,3)=(1,2,3), R_4(1,2)=(1,3,2)$.\

$5)$ $R_5(x)=x^2$ for all $x \in I $ and $R_5(2,3)=(1,3,2), R_5(1,3)=e, R_5(1,2)=(1,2,3)$.\

$6)$ $R_6(x)=x^2$ for all $x \in I $ and $R_6(2,3)=(1,2,3), R_6(1,3)=(1,3,2), R_6(1,2)=e$.\

$7)$ $R_7(x)=x^2$ for all $x \in I $ and $R_7(y)=y$  for all $y \notin I$.
\end{example}
\begin{example}
Let $D_4$ be the dihedral group of order $8$ as a subgroup of $S_4$ that is $D_4=\langle(1,2,3,4), (1,4)(2,3) \rangle$.  We have $I=\langle(1,2,3,4) \rangle$ and $H=D_4/I \cong $ cyclic group of order $2$. There are total $52$ Rota-Baxter operators on $D_4$. We list some of them which are extension of homomorphism on $I$.   

$1)$ $R_1(x)=e$ for all $x \in I $ and $R_1(y)=(1,3)(2,4)$ for all $y \notin I$.\

$2)$ $R_2(x)=e$ for all $x \in I $ and $R_2(y)=(2,4)$ for all $y \notin I$.\

$3)$ $R_3(x)=e$ for all $x \in I $ and $R_3(y)=(1,3)$ for all $y \notin I$.
\end{example}

\begin{example}
Let $Q_8=\langle (1,2,3,4)(5,6,7,8), (1,5,3,7)(2,8,4,6)  \rangle $ be the quaternion group of order $8$ as a subgroup of $S_8$. We have computed that $Q_8$ have total $8$ Rota-Baxter operators.  Let $I=\langle (1,2,3,4)(5,6,7,8) \rangle$, we list few which are extension of  homomorphism on $I$.

$ 1)$ $R_1(x)=e $ for all $x \in I$ and $R_1(y)=(1,3)(2,4)(5,7)(6,8)$ $y \notin I$.\

$ 2)$ $R_2(x)=x $ for all $x \in Q_8$ except $x=(1,6,3,8)(2,5,4,7), (1,8,3,6)(2,7,4,5)$ and 
 $R_2((1,6,3,8)(2,5,4,7))=  (1,8,3,6)(2,7,4,5), R_2((1,8,3,6)(2,7,4,5))=(1,6,3,8)(2,5,4,7)$.\
\end{example}

\section{Fundamental sequence of Well's}
 Given a  Rota-Baxter extension $\mathcal{E} : 0  \rightarrow I \stackrel{i}{\rightarrow} E \stackrel{\pi}{\rightarrow} H \rightarrow 0,$ here we establish various group actions on the set  $\Ext_{\mu}(H, I)$ and construct an exact sequence connecting certain automorphism groups of $(E, R_E)$ with $\Ho_{RBE}^2(H, I)$.

Let $\Aut(H, R_H)$ denotes the automorphism group of the Rota-Baxter group $(H, R_H)$. We define an  action of  $\Aut(H, R_H) \times \Aut(I, R_I)$ on $\Ext(H, I)$ as follows. For a pair  $(\phi, \theta) \in \Aut(H, R_H) \times \Aut(I, R_I)$ of Rota-Baxter automorphisms,  we  define a new extension
$$\mathcal{E}^{(\phi, \theta)} : 0 \rightarrow I \stackrel{i\theta}{\longrightarrow} E \stackrel{\phi^{-1} \pi}{\longrightarrow} H \rightarrow 0.$$ 
Let 
$\mathcal{E}_1:  0  \rightarrow I \stackrel{i}{\rightarrow} E_1 \stackrel{\pi}{\rightarrow}  H \rightarrow 0$ and $\mathcal{E}_2:  0  \rightarrow I \stackrel{i'}{\rightarrow} E_2 \stackrel{\pi'}{\rightarrow}  H \rightarrow 0$
be two equivalent extensions of $(H, R_H)$ by $(I, R_I)$. Then it is not difficult to show that  the extensions $\mathcal{E}_1^{(\phi, \theta)}$ and $\mathcal{E}_2^{(\phi, \theta)}$ are also equivalent for any $(\phi, \theta) \in \Aut(H, R_H) \times \Aut(I, R_I)$. Thus, for a given $(\phi, \theta) \in \Aut(H, R_H) \times \Aut(I, R_I)$, we can define a map from $\Ext(H, I)$ to itself given by 
\begin{equation}\label{act1 sb}
[\mathcal{E}] \mapsto  [ \mathcal{E}^{(\phi, \theta)}].
\end{equation}
If $\phi$ and $\theta$ are identity automorphisms, then obviously $\mathcal{E}^{(\phi, \theta)} = \mathcal{E}$. It is also easy to see that  
$$[\mathcal{E}]  ^{(\phi_1, \theta_1) (\phi_2, \theta_2)}=  \big([\mathcal{E}]^{(\phi_1, \theta_1)}\big)^{(\phi_2, \theta_2)}.$$
We conclude that the association \eqref{act1 sb} gives an action of the group $\Aut(H, R_H) \times \Aut(I,  R_I)$  on the set $\Ext(H, I)$.

As we know that $$\Ext(H, I) = \bigsqcup_{\bar{\mu}} \Ext_{\bar{\mu}}(H, I).$$

 \emph{Let $(I, R_I)$ be a $(H, R_H)$-module by a  fixed action  $\mu : H \rightarrow \Aut(I) $.}  Let $\C_{\mu}$ denote the stabilizer of $\Ext_{\mu}(H, I)$ in $\Aut(H, R_H) \times \Aut(I, R_I)$. More explicitly,
$$\C_{\mu} = \{ (\phi, \psi) \in \Aut(H, R_H) \times \Aut(I, R_I) \mid \mu_h=\psi^{-1} \mu_{\phi (h)} \psi \}.$$
Notice that  $\C_\mu$ is a subgroup of $\Aut(H, R_H) \times \Aut(I, R_I)$, and it  acts on  $\Ext_{\mu}(H, I)$ by the same rule as given in \eqref{act1 sb}. \

Next we consider an action of $ \C_{\mu}$ on $\Ho^2_{RBE}(H, I)$.
Let $(\phi, \theta) \in \Aut(H, R_H) \times \Aut(I, R_I)$ and $f \in \Fun(H^n, I)$, where $n \ge 1$ be an integer. Define $f^{(\phi, \psi)} : H^n \to I$ by setting
$$f^{(\phi, \psi)}(h_1, h_2, \ldots, h_n) := \psi^{-1}\big(f(\phi(h_1), \phi(h_2), \ldots, \phi(h_n))\big).$$ 
It is not difficult to see that the group $\Aut(H) \times \Aut(I)$ acts on  the group $\Fun(H^n, I)$ as well as on the group $C^{n}(H, I)$, by automorphisms, given by the association
\begin{equation}\label{act2 sb}
f \mapsto f^{(\phi, \psi)}.
\end{equation}
It is also obvious that  $\C_{\mu}$ acts on both of these sets. We are interested in the action of $\C_{ \mu}$ on $\Ho_{RBE}^2(H, I)$. The association \eqref{act2 sb} induces an action of $\C_{\mu}$ on $TC^2_{RBE} = C^2(H, I) \oplus C^1(H, I)$ by setting 
\begin{equation}\label{act3 sb}
(\tau, g) \mapsto \big(\tau^{(\phi, \psi)}, g^{(\phi, \psi)}\big).
\end{equation}
\begin{lemma}\label{lemma-act3 sb}
For $(\phi, \psi) \in \C_{\mu}$,  the following hold:

$(i)$ If $(\tau, g) \in \Z^2_{RBE}(H,I)$, then $\big(\tau^{(\phi, \psi)}, g^{(\phi, \psi)}\big) \in \Z^2_{RBE}(H,I)$.

$(ii)$ If $(\tau, g) \in \B^2_{RBE}(H,I)$, then $\big(\tau^{(\phi, \psi)}, g^{(\phi, \psi)}\big) \in \B^2_{RBE}(H,I)$.

Hence, the association \eqref{act3 sb} gives an action of $\C_{\mu}$ on $\Ho_{RBE}^2(H, I)$ by automorphisms, if we define
$$[(\tau, g)]^{(\phi, \psi)} =[\big(\tau^{(\phi, \psi)}, g^{(\phi, \theta)}\big)].$$
\end{lemma}
\begin{proof} 
$(i)$ It is easy to see that $\tau^{(\phi, \psi)}$ satisfy \eqref{cocycle 2}. Replace $h_1, h_2$ in  \eqref{2-cocycle condn} by $\phi(h_1), \phi(h_2)$ and applying $\psi^{-1}$ both side we have
\begin{align*}\label{abelianparent2}
& \psi^{-1}\Big(g(\phi(h_2))-g(\phi(h_1 \circ_{R_H} h_2)) + \mu_{R_H(\phi(h_2))}(g(\phi(h_1))) -R_I \big(\mu_{\phi( h_1 \circ h_2 R_H(h)^{-1})}(\mu_{h_2}(g(h_1))- g(h_1)\big)\Big) \notag \\
 =&\psi \Big( \tau(R_H(\phi(h_1)), R_H(\phi(h_2)))-R_I \big(\mu_{\phi(h_1 \circ h_2)}(\tau(\phi(h_1R_H(h_1)), \phi(h_2R_H(h_1))^{-1})\notag \\
 &+\mu_{\phi(h_2 R_H(h_1)^{-1})}(\tau(\phi(h_1), R_H(\phi(h_1)))
  + \tau(\phi(h_2), R_H(\phi(h_1))^{-1})-\tau(R_H(\phi(h_1)), R_H(\phi(h_1))^{-1})\big)\Big).
\end{align*}
Now using that $\psi$ and $\phi$ commutes with the operator $R_H$ and $R_I$, respectively and definition of $\C_{\mu}$ we easily see that  $\big(\tau^{(\phi, \psi)}, g^{(\phi, \psi)}\big)$ satisfy \eqref{2-cocycle condn}.

$(ii)$ Let $\tau= \delta^1(\theta)$ and $g=-\Phi(\theta)$ for some $\theta \in TC^1_{RBE}$.
We have 
\begin{align*}
 \tau^{(\phi, \psi)}(h_1, h_2) &=\psi^{-1} \big(\theta(\phi(h_2))-\theta(\psi(h_1 \circ_{R_H} h_2))+\mu_{\phi(R_H(h_2))}(\theta(\phi(h_1))) \big) \delta^1(\theta^{(\phi, \psi)})(h_1, h_2).
\end{align*}
Similarly,  we have
\begin{align*}
g^{(\phi, \psi)}= -\Phi(\theta^{(\phi, \psi)}).
\end{align*} 
\hfill $\Box$

\end{proof}
\begin{remark}
The action of $\C_{\mu}$  on $\Ho^2_{RBE}(H,I)$ as defined in the preceding lemma, can be transferred  on  
$\Ext_{\mu}(H,I)$ through the bijection given in  Theorem  \ref{bij}.  Notice that the resulting action of $\C_{\mu}$ on $\Ext_{\mu}(H,I)$ agrees with the action defined in \eqref{act1 sb}.
\end{remark}

We now consider the action of  $\Ho^2_{RBE}(H,I)$ onto itself by right translation, which is faithful and transitive. Again using Theorem \ref{bij}, we can transfer this action on $\Ext_{\mu}(H,I) = \{[\mathcal{E}(\tau, g)] \mid [(\tau, g)] \in \Ho^2_{RBE}(H,I)\}$. More precisely, for $[(\tau_1, g_1)] \in \Ho^2_{RBE}(H,I)$, the action is given by
$$[\mathcal{E}(\tau, g)]^{[(\tau_1, g_1)]} =[ \mathcal{E}(\tau + \tau_1, g + g_1)],$$
for all  $\mathcal{E}(\tau, g) \in \Ext_{\mu}(H,I)$. Notice that this action is faithful and transitive.
Consider the semi-direct product $\Gamma :=\C_{\mu} \ltimes \Ho^2_{RBE}(H,I)$  under  the action defined in Lemma \ref{lemma-act3 sb}. We wish to define an action of $\Gamma$ on $\Ext_{\mu}(H,I)$.  For $(c, h) \in \Gamma$ and $[\mathcal{E}] \in \Ext_{\mu}(H,I)$, define 
\begin{equation}\label{act4 sb}
[\mathcal{E}]^{(c, h)} = ([\mathcal{E}]^c)^h.
\end{equation}

\begin{lemma}\label{wells2 sb}
The rule  in \eqref{act4 sb}   defines an action of  $\Gamma$ on $\Ext_{ \mu}(H,I)$.
\end{lemma}
\begin{proof}
 Notice that for $(c_1,h_1),  (c_2, h_2) \in \Gamma$,  $(c_1,h_1)(c_2,h_2)=(c_1c_2, h_1^{c_2} \, h_2)$.  So, it is enough to show that $\big([\mathcal{E}]^h\big)^c = \big([\mathcal{E}]^c\big)^{h^c}$ for each $c \in \C_{\mu}$, $h \in \Ho^2_{RBE}(H,I)$ and $[\mathcal{E}] \in \Ext_{\mu}(H,I)$.
We know that  $[\mathcal{E}]=[(H,I,\mu,\tau, g)]$ for some $[(\tau, g)] \in \Ho^2_{RBE}(H,I)$.  Then, for $h=[(\tau_h, g_h)] \in \Ho^2_N(H, I)$, we have
\begin{eqnarray*}
\big([\mathcal{E}]^{h}\big)^c &=&[\mathcal{E}((\tau+\tau_h)^c), (g+g_h)^c)]\\
&=& [\mathcal{E}(\tau^c+ \tau_h^c, g^c +g_{h}^c)]\\
&=& ([\mathcal{E}(\tau^c, g^c)])^{h^c}\\
&=& \big([\mathcal{E}]^c\big)^{h^c}.
\end{eqnarray*}
The proof is now complete. \hfill $\Box$

\end{proof}
\vspace{-.4cm}
 Let $[\mathcal{E}] \in \Ext_{\mu}(H,I)$ be a fixed extension. Since the action of $\Ho^2_{RBE}(H,I)$ on $\Ext_{\mu}(H,I)$ is transitive and faithful, for  each $c \in \C_{\mu}$, there exists a unique element (say) $h_c$  in  $\Ho^2_{RBE}(H,I)$ such that  
 $$[\mathcal{E}]^{c} = [\mathcal{E}]^{h_c}.$$
Thus, we have a well-defined map $ \omega(\mathcal{E}): \C_{\mu} \rightarrow \Ho^2_{RBE}(H,I)$ given by
 \begin{equation}\label{wells-map sb}
 \omega(\mathcal{E})(c)=h_c,  \hspace{.2cm}\mbox{ for} \hspace{.2cm} c \in \C_{\mu}.
 \end{equation}
  
\begin{lemma}\label{wells3 sb}
The map $ \omega(\mathcal{E}): \C_{ \mu} \rightarrow \Ho^2_{RBE}(H,I)$ given in \eqref{wells-map sb} is a derivation with respect to the action of $\C_{\mu}$ on $H^2_{RBE}(H,I)$ given in \eqref{act3 sb}.
\end{lemma}
\begin{proof}
Let $c_1, c_2 \in \C_{\mu}$ and  $\omega(\mathcal{E})(c_1c_2) = h_{c_1c_2}$.  Thus, by the definition of $\omega(\mathcal{E})$,   $[\mathcal{E}]^{c_1c_2} = [\mathcal{E}]^{h_{c_1c_2}}$.  Using the fact that  $\big([\mathcal{E}]^{h}\big)^{c} = \big([\mathcal{E}]^{c}\big)^{h^c}$ for each $c \in \C_{(\nu, \mu, \sigma)}$, $h \in \Ho^2_{RBE}(H,I)$, we have
\begin{eqnarray*}
 [\mathcal{E}]^{h_{c_1c_2}} & = &[\mathcal{E}]^{(c_1c_2)}\\
 &=& \big([\mathcal{E}]^{c_1}\big)^{c_2}\\
 &=& \big([\mathcal{E}]^{h_{c_1}}\big)^{c_2}\\
 &= & \big([\mathcal{E}]^{c_2}\big)^{(h_{c_1})^{c_2}}\\
 &=& \big([\mathcal{E}]^{h_{c_2}}\big)^{(h_{c_1})^{c_2}}\\
 &=& [\mathcal{E}]^{\big(h_{c_2} + (h_{c_1})^{c_2} \big)}.
 \end{eqnarray*}
 Since the action of $\Ho^2_{RBE}(H,I)$ on $\Ext_{\mu}(H,I)$  is faithful,  it follows that $h_{c_1c_2} = (h_{c_1})^{c_2} + h_{c_2}$. This implies that $\omega(\mathcal{E})(c_1c_2)=\big(\omega(\mathcal{E})(c_1)\big)^{c_2}+\omega(\mathcal{E})(c_2)$ which shows that  $\omega(\mathcal{E})$ is a derivation. \hfill $\Box$
 
\end{proof}
Although $\omega(\mathcal{E})$ is not a homomorphism, but we can still talk about its set theoretic kernel, that is,
$$\Ker(\omega(\mathcal{E})) = \{c \in C_{\mu} \mid [\mathcal{E}]^c=[\mathcal{E}]\}.$$

Let $\mathcal{E}: 0 \rightarrow I \rightarrow E \overset{\pi}\rightarrow H \rightarrow 0$
be an extension of a Rota-Baxter group $(H, R_H)$ by an abelian Rota-Baxter group $(I, R_I)$ such that $[\mathcal{E}] \in \Ext_{\mu}(H,I)$.
Let  $\Aut_I(E, R_E)$ denote the subgroup of $\Aut(E, R_E)$ consisting of all automorphisms of $E$ which normalize $I$, that is,
$$\Aut_I(E, R_E) := \{ \gamma \in \Aut(E, R_E) \mid \gamma(y) \in I \mbox{ for all }  y \in I\}.$$ 
For $\gamma \in \Aut_I(E, R_E)$, set $\gamma_I := \gamma |_I$, the restriction of $\gamma$ to $I$, and $\gamma_H$ to be the automorphism of $(H, R_H)$ induced by $\gamma$. More precisely, $\gamma_H(h) = \pi(\gamma(s(h)))$ for all $h \in H$, where $s$ is a st-section of $\pi$. Notice that the definition of $\gamma_H$ is independent of  the choice of a st-section. Define a group homomorphism $\rho(\mathcal{E}) :  \Aut_I(E, R_E) \rightarrow  \Aut(H, R_H) \times \Aut(I, R_I)$ by
$$\rho(\mathcal{E})(\gamma)=(\gamma_H, \gamma_I).$$ 

\begin{prop}\label{wells4 sb}
For the extension $\mathcal{E}$, we have   $\IM(\rho(\mathcal{E})) = \Ker(\omega(\mathcal{E})) \subset \C_{\mu}$.
\end{prop}
 \begin{proof}
 We first  show that $\IM(\rho(\mathcal{E})) \subset \C_{\mu}$. To show this, we need to prove  that $\mu_{h} = \gamma_I^{-1} \mu_{\gamma_H(h)} \gamma_I$ for all $h \in H$. Let $s$ be an st-section of $\pi$  and $ x \in E$. Notice that $\gamma_I^{-1}$ is the restriction of $\gamma^{-1}$ on $I$.   Also notice that  for a given $x \in E$, $s(\pi(h)) = h y_h$ for some $y_h \in I$.  Now for $h \in H$ and $ y \in I$, we have
\begin{eqnarray*}
\gamma_I^{-1} \mu_{\gamma_H(h)} \gamma_I(y) &=& \gamma^{-1} \big(\mu_{\pi(\gamma(s(h)))}(\gamma(y))\big)\\
&=& \gamma^{-1}\big(s(\pi(\gamma(s(h))))^{-1} \gamma (y)  s(\pi(\gamma(s(h))))\big)\\
&=& \gamma^{-1}\big( (\gamma(s(h)) y_{\gamma(s(h))})^{-1}  \gamma(y) \gamma(s(h)) y_{\gamma(s(h))} \big)\\
&=& (s(h) \gamma^{-1}( y_{\gamma(s(h))}))^{-1} y s(h) \gamma^{-1}( y_{\gamma(s(h))}) \\
&=& \gamma^{-1}( y_{\gamma(s(h))})^{-1} s(h)^{-1} y s(h) \gamma^{-1}( y_{\gamma(s(h))}) \\
&=&  \gamma^{-1}( y_{\gamma(s(h))})^{-1} \mu_h(y) \gamma^{-1}( y_{\gamma(s(h))}) \\
&=& \mu_h(y).
\end{eqnarray*}
Hence, $\mu_{h} = \gamma_I^{-1} \mu_{\gamma_H(h)} \gamma_I$.  
Now, we prove that $\IM(\rho(\mathcal{E})) = \Ker(\omega(\mathcal{E}))$.

Let $\rho(\mathcal{E})(\gamma) =  (\gamma_H, \gamma_I)$ for  $\gamma \in \Aut_I(E, R_E)$.  We know that $s(\pi(x))=x y_x $ for some $y_x \in I$.  Thus, we have
\begin{eqnarray*}
\gamma_H^{-1}(\pi(\gamma(s(h)))) &=& \pi\big(\gamma^{-1}(s(\pi(\gamma(s(h)))))\big)\\
&=& \pi\big(\gamma^{-1}\big(\gamma(s(h)) y_{\gamma(s(h))}\big)\big)\\
&= & h,
\end{eqnarray*}
 which implies that the diagram commutes
$$\begin{CD}
0 @>>> I @>>> E@>{{\pi} }>> H @>>> 0\\ 
&& @V{\text{Id}} VV @V{\gamma} VV @V{\text{Id} }VV \\
0 @>>> I @>{\gamma_I}>> E @>{\gamma_H^{-1} \pi}>> H  @>>> 0.
\end{CD}$$
Hence $[(\mathcal{E})]^{(\gamma_H, \gamma_I)} =[(\mathcal{E})]$, which shows that $\IM(\rho(\mathcal{E})) \subseteq \Ker(\omega(\mathcal{E}))$. 

Conversely, if $(\phi, \psi) \in  \Ker(\omega(\mathcal{E}))$,  then there exists a Rota-Baxter group homomorphism $\gamma: (E, R_E) \rightarrow (E, R_E)$ such that the diagram commutes
$$\begin{CD}
0 @>>> I @>>> E@>{{\pi} }>> H @>>> 0\\ 
&& @V{\text{Id}} VV@V{\gamma} VV @V{\text{Id} }VV \\
0 @>>> I @>{\psi}>> E @>{ \phi^{-1} \pi}>> H  @>>> 0.
\end{CD}$$
 It is now obvious that $\gamma \in \Aut_I(E, R_E)$, $\psi = \gamma_I$ and $\phi = \gamma_H$. Hence
 $\rho(\mathcal{E})(\gamma)=(\phi, \psi)$, which completes the proof.  \hfill   $\Box$

\end{proof}
\vspace{-.5cm}
Continuing with the above setting, set $\Aut^{H, I}(E, R_E) := \{\gamma \in \Aut(E, R_E) \mid \gamma_I = \Id,  \gamma_H = \Id\}$. Notice that  $\Aut^{H,I}(E, R_E)$ is precisely the kernel of $\rho(\mathcal{E})$. Hence, by using Proposition \ref{wells4 sb}, we get the following.

\begin{thm}\label{wells5 sb}
Let $\mathcal{E}: 0 \rightarrow I \rightarrow E \overset{\pi}\rightarrow H\rightarrow 0$ be an extension of a Rota-Baxter group $(H, R_E)$ by an abelian Rota-Baxter group $(I, R_I)$ such that $[\mathcal{E}] \in \Ext_{ \mu}(H,I)$. Then we have the following exact sequence of groups 
$$0 \rightarrow \Aut^{H,I}(E, R_E) \rightarrow \Aut_I(E, R_E) \stackrel{\rho(\mathcal{E})}{\longrightarrow} \C_{\mu} \stackrel{\omega(\mathcal{E})}{\longrightarrow} \Ho^2_{RBE}(H,I),$$
where $\omega(\mathcal{E})$ is, in general,  only a derivation.
\end{thm}

In the following result we give a new interpretation of the group $\Aut^{H, I}(E, R_E)$.
\begin{prop}\label{wells6 sb}
Let  $\mathcal{E} :  0 \rightarrow I \rightarrow E \overset{\pi}\rightarrow H \rightarrow 0$ be a Rota-Baxter  extension of $(H, R_H)$ by $(I, R_I)$ such that $[\mathcal{E}] \in \Ext_{\mu}(H, I)$.  Then  $\Aut^{H,I}(E, R_E) \cong \Z^1_{RBE}(H,I)$.
\end{prop}
\begin{proof}
We know that every element $x \in E$ has a unique expression of the form $x=s(h)y$ for some $h \in H$ and $y \in I$. Let us define a map $\eta: \Z^1_{RBE}(H,I)  \rightarrow  \Aut^{H,I}(E, R_E)$ by 
$$\eta(\lambda)((s(h)y) = s(h) \lambda(h) y,$$
where $\lambda \in \Z^1_{RBE}(H,I)$.
Notice that the image of $\eta(\lambda)$ is independent of the choice of a st-section.
We claim that $\eta(\lambda) \in  \Aut^{H,I}(E, R_E)$.  It follows from the theory of group extensions that $\eta(\lambda)$ is an automorphism of $E$ and it fixes $I$ and $H$.  We just check that $\eta(\lambda)$ commutes with the operator $R_E$. For $h \in H$ and $y_1 \in I$, we have
\begin{align*}
\eta(\lambda) (R_E(s(h)y))&=\eta(\lambda)\big(s(R_H(h)) y_h R_I(\mu_{R_H(h)}(y))\big)\\
&=s(R_H(h)) \lambda(R_H(h)) y_h  R_I(\mu_{R_H(h)}(y))\\
&=s(R_H(h)) R_I(\mu_{R_H(h)}(\lambda(h))) y_h  R_I(\mu_{R_H(h)}(y))\\
&=s(R_H(h))y_h R_I(\mu_{R_H(h)}(\lambda(h)y))  \\
&=R_E\big(s(h)\lambda(h)y\big)\\
&=R_E( \eta(\lambda)(s(h)y))
\end{align*}
which shows that $\eta(\lambda) \in  \Aut^{H,I}(E, R_E)$.

 We will now define a map $\zeta :  \Aut^{H,I}(E, R_E)  \rightarrow \Z^1_{RBE}(H,I)$ as follows. For $\gamma \in \Aut^{H,I}(E, R_E)$ and  $h \in H$, there exists a unique element (say) $y^{\gamma}_h \in I$ such that  $\gamma(s(h)) = s(h) y^{\gamma}_h$. Thus, for $h \in H$,  define $\zeta$ by 
 $$\zeta(\gamma)(h) = y^{\gamma}_h.$$
 Notice that $\zeta(\gamma)$ is independent of the choice of a st-section. Further, it is  a group-theoretic derivation that  follows from the theory of group extensions. We check that it satisfies
\begin{eqnarray*}
\gamma(R_E(s(h))) &=& R_E(\gamma(s(h))),\\
\gamma(s(R_H(h)) y_h)&=& R_E(s(h)y^{\gamma}_h),\\
s(R_H(h))y^{\gamma}_{R_H(h)} y_h&=&s(R_H(h))y_h R_I(\mu_{R_H(h)}(y^{\gamma}_h)).
\end{eqnarray*}
Hence, we have $\zeta(\gamma)(R_H(h))=R_I(\mu_{R_H(h)}(\zeta(\gamma)(h)))$. We have now proved that  $\zeta(\gamma) \in  \Z^1_{RBE}(H,I)$. Both $\eta$ and $\zeta$ are homomorphisms, and $\eta \zeta$ and $\zeta \eta$  are identity on  $\Autb^{H,I}(E)$ and $\Z^1_N(H,I)$, respectively, is obvious.  Hence $\Autb^{H,I}(E) \cong  \Z^1_N(H,I)$, and the proof is complete. \hfill $\Box$

 \end{proof}
 \vspace{-.5cm}
We finally get the following Well's like exact sequence for Rota-Baxter groups.
\begin{thm}\label{wells7 sb}
Let $\mathcal{E}: 0 \rightarrow I \rightarrow E \overset{\pi}\rightarrow H \rightarrow 0$ be an extension of a Rota-Baxter group $(H, R_H)$ by an abelian Rota-Baxter group $(I, R_I)$ such that $[\mathcal{E}] \in \Ext_{\mu}(H,I)$. Then we have the following exact sequence of groups 
$$0 \rightarrow \Z^1_{RBE}(H,I) \rightarrow \Aut_I(E, R_E) \stackrel{\rho(\mathcal{E})}{\longrightarrow} \C_{(\mu,\sigma, \nu)} \stackrel{\omega(\mathcal{E})}{\longrightarrow} \Ho^2_{RBE}(H,I),$$
where $\omega(\mathcal{E})$ is, in general, only a derivation.
\end{thm}

\end{document}